\documentclass[11pt]{article}
\usepackage{geometry}
\usepackage{amsmath,amssymb, array}
\usepackage{epstopdf}
\usepackage[all]{xy}

\usepackage{amsthm}
\usepackage{color}
\usepackage[dvips]{graphicx}
\title{Log Rigid Syntomic Cohomology\\for Strictly Semistable Schemes}
\author{Kazuki Yamada}

  \makeatletter
  \@addtoreset{equation}{section}
  \makeatother

\DeclareMathOperator{\Spec}{\mathrm{Spec}}
\DeclareMathOperator{\Proj}{\mathrm{Proj}}

\DeclareMathOperator{\Spwf}{\mathrm{Spwf}}
\DeclareMathOperator{\Hom}{\mathrm{Hom}}
\DeclareMathOperator{\Ext}{\mathrm{Ext}}
\DeclareMathOperator{\iHom}{\underline{\mathrm{Hom}}}
\DeclareMathOperator{\cHom}{\underline{\mathrm{Hom}}^{\bullet}}
\DeclareMathOperator{\Cone}{\mathrm{Cone}^{\bullet}}
\DeclareMathOperator{\Ker}{\mathrm{Ker}}
\DeclareMathOperator{\Coker}{\mathrm{Coker}}
\DeclareMathOperator{\cKer}{\mathrm{Ker}^{\bullet}}
\DeclareMathOperator{\cCoker}{\mathrm{Coker}^{\bullet}}
\def\Sm{(\mathrm{Sm}_K)^{\mathrm{op}}}
\def\MM{\mathcal{MM}_K}
\def\Sh{\mathrm{Sh}}
\def\rG{R\Gamma}
\def\RG{\mathbb{R}\Gamma}
\def\Vec{\mathrm{Vec}}

\def\log{\mathrm{log}}
\def\Ab{\mathrm{Ab}}

\def\cZ{\mathbb{Z}}

\def\cN{\mathbb{N}}
\def\cQ{\mathbb{Q}}
\def\cC{\mathbb{C}}

\def\dR{\mathrm{dR}}
\def\rig{\mathrm{rig}}
\def\Hdg{\mathrm{Hdg}}

\def\syn{\mathrm{syn}}
\def\Gd{\mathrm{Gd}}
\def\Vec{\mathrm{Vec}}
\def\an{\mathrm{an}}

\def\pHC{p\mathrm{HC}_K}
\def\pHK{p\mathrm{HK}_K}
\def\pHD{p\mathrm{HD}_K}
\def\MF{\mathrm{MF}^{\mathrm{ad}}_K(\phi,N)}
\def\MFG{\mathrm{MF}^{\mathrm{ad}}_K(\phi,N,G_K)}
\def\MFC{\mathrm{MF}^{\mathrm{ad}}_K(\phi)}
\def\incl{\ar@{^{(}-{>}}}
\def\lincl{\ar@{_{(}-{>}}}
\def\comm{\ar@{}|\circlearrowleft}

\DeclareMathOperator{\Tot}{\mathrm{Tot}}
\DeclareMathOperator{\Ass}{\mathrm{Ass}^{\bullet}}
\DeclareMathOperator{\Ind}{\mathrm{Ind}}
\def\GL{\mathrm{GL}}
\def\BGL{\mathrm{B}_{\bullet}\mathrm{GL}}

\theoremstyle{definition}
\newtheorem{theorem}{Theorem}[subsection]
\newtheorem{proposition}[theorem]{Proposition}
\newtheorem{lemma}[theorem]{Lemma}

\newtheorem{definition}[theorem]{Definition}

\newtheorem{remark}[theorem]{Remark}

\begin{document}
\maketitle
{\small
\begin{center}
ABSTRACT
\end{center}
\vspace{-10pt}
We construct log rigid syntomic cohomology for strictly semistable schemes over the ring of integers of a $p$-adic field, and prove that it is interpreted as the extension group of the complex of admissible filtered $(\phi,N)$-modules.}
\tableofcontents
\section{Introduction}
\subsection{Main result}
Let $K$ be a $p$-adic field with the ring of integers $V$. For a smooth scheme $\mathcal{X}$ over $V$, the rigid syntomic cohomology group $\widetilde{H}_{\syn}^i(\mathcal{X},n)$ was defined firstly by Gros \cite{Gros} under certain assumptions, and then by Besser \cite{Besser} in general. Note that the notation $\widetilde{H}_{\syn}$ is different from \cite{Gros} and \cite{Besser}. Bannai \cite{Bannai} proved that the rigid syntomic cohomology can be interpreted as the absolute $p$-adic Hodge cohomology. His construction is improved by Chiarellotto, Ciccioni, and Mazzari \cite{CCM} using generalized Godement resolution.

Absolute cohomology means the extension group of geometric cohomology complex. See Table \ref{table} of the bottom of this page for example.
\begin{table}[b]\caption{geometric and absolute cohomology}\label{table}{\footnotesize
\begin{tabular}{|c|c|c|c|}\hline
$X$&geometric cohomology&absolute cohomology&realization category\\ \hline \hline
\begin{tabular}{c}over a field\\ of char. $\neq \ell$\end{tabular}&$\ell$-adic \'{e}tale of $\overline{X}$&$\ell$-adic \'{e}tale of $X$&{\begin{tabular}{c}$\ell$-adic representations of\\ the absolute Galois group\end{tabular}}\\ \hline
\begin{tabular}{c}smooth\\ over $\cC$\end{tabular}&{\begin{tabular}{c}Hodge\\ \rotatebox{90}{$=$}\\ Betti $+$ de Rham\end{tabular}}&{\begin{tabular}{c}absolute Hodge\\ \rotatebox{90}{$=$}\\ Deligne\end{tabular}}&mixed Hodge structures\\ \hline
{\begin{tabular}{c}smooth\\ over $V$\end{tabular}}&{\begin{tabular}{c}$p$-adic Hodge\\ \rotatebox{90}{$=$}\\ rigid $+$ de Rham\end{tabular}}&{\begin{tabular}{c}absolute $p$-adic Hodge\\ \rotatebox{90}{$=$}\\ rigid syntomic\end{tabular}}&{\begin{tabular}{c}admissible filtered\\ $(\phi,N)$-modules\end{tabular}}\\ \hline
\end{tabular}}\end{table}
Let $\mathrm{Sm}_K$ be the category of smooth varieties over $K$. Conjecturally, there should exist the rigid abelian tensor category $\MM$ of mixed motives over $K$ with a functor $h:\Sm\rightarrow D^b(\MM)$, satisfying some desirable properties. For example, for any cohomology theory $*$ satisfying some axioms with an abelian category $\mathcal{A}$ and a functor $\rG_*:\Sm\rightarrow D^b(\mathcal{A}_*)$, there is the realization functor $\mathcal{R}_*:\MM\rightarrow\mathcal{A}_*$ such that its derived functor $D\mathcal{R}_*:D^b(\MM)\rightarrow D^b(\mathcal{A}_*)$ satisfies $\rG_*=D\mathcal{R}_*\circ h$. Then $D\mathcal{R}_*$ induces the map
\begin{eqnarray}
D\mathcal{R}_*:\Ext^i_{D^b(\MM)}(1,h(X)(n))&\rightarrow&\Ext^i_{D^b(\mathcal{A}_*)}(1,D\mathcal{R}_*\circ h(X)(n))\\ \label{realization}
\nonumber&=&\Ext^i_{D^b(\mathcal{A}_*)}(1,\rG_*(X)(n)).
\end{eqnarray}
The motivic cohomology $H^i_{\mathcal{M}}(X,n)$ and the absolute cohomology $H_{\mathrm{abs}-*}^i(X,n)$ for $*$ are defined to be the left and right hand side of (1.1) respectively.
Note that the triangulated category having desired properties which $D^b(\MM)$ should have is constructed by Voevodsky \cite{Voevodsky}. For example the equality
\[\Ext^i_{D^b(\MM)}(1,h(X)(n))=K_{2n-i}(X)_{\cQ}^{(n)},\]
holds, where the right hand side is the eigenspace of Adams operation on the rational $K$-group.

Absolute cohomology should be related with the values of ($p$-adic) $L$-functions in the context of ($p$-adic) Beilinson conjectures. For smooth projective varieties having good reductions at $p$, $p$-adic Beilinson conjecture was formulated by Perrin-Riou \cite{PR}. It can be interpreted through the rigid syntomic regulator map
\[r_{\syn}:K_i(\mathcal{X})\rightarrow \widetilde{H}^{2n-i}_{\syn}(\mathcal{X},n)\]
constructed by Besser \cite {Besser}, and proved in some special cases (cf. \cite{BBdJR}, \cite{BK}). 
Conjecturally, $r_{\syn}$ should commute with $D\mathcal{R}_{\Hdg}$ in the sense of section \ref{philosophy}.
In the general setting including bad reduction, $p$-adic Beilinson conjecture has not yet been formulated.

The purpose of this paper is to define the log rigid syntomic cohomology for strictly semistable schemes as the absolute $p$-adic Hodge cohomology.
For a strictly semistable scheme $\mathcal{X}$, Gro\ss e-Kl\"{o}nne \cite{Grosse1} introduced the log rigid cohomology of the special fiber on which the Frobenius operator and monodromy operator act. On the other hand, the de Rham cohomology of the generic fiber has the Hodge filtration. Gluing these complexes, we get the $p$-adic Hodge complex $\RG_{\Hdg}(\mathcal{X})$ which is an object of the certain triangulated category $\pHC$. The derived category $\pHD$ of $\pHC$ exists, and we define the log rigid syntomic cohomology group with Tate $n$-twist by
\[H^i_{\syn}(\mathcal{X},n)=\Ext^{i}_{\pHD}(K_0,\RG_{\Hdg}(\mathcal{X})(n)).\]
The equivalence $\Theta:D^b(\MF)\rightarrow\widetilde{\pHD}$ in Theorem \ref{cateqthm} gives a validity of this definition. Namely, when $\mathcal{X}$ satisfies Hyodo-Kato condition (HK) in section \ref{section4.1}, $H^i_{\syn}(\mathcal{X},n)$ is regarded as the extension group of a complex of admissible filtered $(\phi,N)$-modules associated to $\mathcal{X}$.

In \cite{NN}, Nekov\'{a}\v{r} and Nizio{\l} defined the syntomic cohomology for varieties over $K$ using crystalline cohomology and $h$-theory. Their results are very general and contain the situation of this paper. An advantage of (log) rigid syntomic cohomology is that (log) rigid cohomology is defined as de Rham-type cohomology of an analytic space, unlike \'{e}tale and $h$-cohomology. So it is hopefully more amenable to explicit calculation, and we will be able to use $p$-adic analysis to relate it with the values of $L$-functions. Since (log) rigid syntomic cohomology should have suitable properties for non-proper schemes, it would be usable to construct the $p$-adic realization of polylogarithms.

\subsection{Philosophy of $p$-adic Hodge cohomology}\label{philosophy}
Let $G_K$ be the absolute Galois group of $K$. Let $\MFC$, $\MF$, and $\MFG$ be the categories of admissible filtered $\phi$-modules, $(\phi,N)$-modules, and $(\phi,N,G_K)$-modules defined by Fontaine. Note that they are equivalent to the categories of crystalline, semistable, and de Rham representations of $G_K$.

For a strictly semistable scheme $\mathcal{X}$ satisfying (HK), $\RG_{\Hdg}(\mathcal{X})$ in $\pHD$ should be depend only on the generic fiber $X$. Moreover there should exist $p$-adic Hodge cohomology theory for smooth varieties over $K$ with
\[\rG_{\Hdg}:(\mathrm{Sm}_K)^{\mathrm{op}}\rightarrow D^b(\MFG)\]
such that $\rG_{\Hdg}(X)$ is in $D^b(\MF)$ and correspond to $\RG_{\Hdg}(\mathcal{X})$ by $\Theta$. Fontaine and Perrin-Riou's observation of cohomology of filtered modules suggests
\[\Ext^i_{D^b(\MFG)}(K^{\mathrm{ur}}_0,\rG_{\Hdg}(X)(n))=\Ext^i_{D^b(\MF)}(K_0,\rG_{\Hdg}(X)(n)),\]
namely
\begin{equation}\label{synabs}
H^i_{\mathrm{abs}-\Hdg}(X,n)=H^i_{\syn}(\mathcal{X},n).
\end{equation}

For a smooth scheme $\mathcal{X}$ satisfying (HK), $\rG_{\Hdg}(X)$ will be in $D^b(\MFC)$, and we should have 
\[\widetilde{H}_{\syn}^i(\mathcal{X},n)=\Ext^i_{D^b(\MFC)}(K_0,\rG_{\Hdg}(X)(n)).\]
Let $H^i_{\mathcal{M}}(X,n)_{\cZ}$ be the integral part of the motivic cohomology of $X$, defined by the image of the canonical map $K_{2n-i}(\mathcal{X})\rightarrow H^i_{\mathcal{M}}(X,n)$. Then $r_{\syn}$ above passes through
\[r_{\syn}:H^i_{\mathcal{M}}(X,n)_{\cZ}\rightarrow\widetilde{H}_{\syn}^i(\mathcal{X},n),\]
and it commutes with
\[D\mathcal{R}_{\Hdg}:H^i_{\mathcal{M}}(X,n)\rightarrow H^i_{\mathrm{abs}-\Hdg}(X,n)\]
through the identification \eqref{synabs} and the canonical projection $H_{\syn}^i(\mathcal{X},n)\rightarrow\widetilde{H}_{\syn}^i(\mathcal{X},n)$ given by Proposition \ref{syncomparison}.

\subsection{Outline of this paper}
In section 2, we recall about admissible filtered $(\phi,N)$-modules and define $p$-adic Hodge complexes. The extension group of an admissible filtered $(\phi,N)$-module by $K_0$ is computed in \cite{FP} as the cohomology group of the certain simple complex. We generalize this calculation to extension groups of a complex of admissible filtered $(\phi,N)$-modules. And then we apply this argument to the calculation of extension groups of $p$-adic Hodge complexes. Using those calculations, we will show that $\Theta:D^b(\MF)\rightarrow\widetilde{pHD}$ is equivalent (Theorem \ref{cateqthm}). These argument follows that of Bannai \cite{Bannai} with the techniques of gluing categories in \cite{Beilinson0} and \cite{Huber}. 

In section 3, we construct the $p$-adic Hodge complex associated to $\mathcal{X}$. The constructions of log rigid complexes and comparison maps of them are almost due to Gro\ss e-Kl\"{o}nne in \cite{Grosse1}. He constructed the log rigid complexes of two types. One is defined by choosing admissible liftings and using Steenbrink double complexes, then it has the Frobenius operator and the monodromy operator. Another one is defined more simply, and related with the de Rham complex directly. He also constructed the comparison maps between them. We will show the functoriality of the log rigid complexes with the operators, which was only unproven.

In section 4, we define the log rigid syntomic cohomology, state some properties concluded immediately. In particular, we construct the Chern class map following the method of \cite{Besser} and \cite{Huber}.

\subsection{Acknowledgments}
The author would like to thank Professor K. Bannai for helpful suggestions and giving me opportunities of positive discussion.
The author also thank members of the KiPAS-AGNT group, S. Ohgaki, and S. Ohkawa for helpful discussions, and all members of the Department of Mathematics of Keio University for their hospitality.
It is also a pleasure to thank Professor A. Shiho for answering kindly to my questions about his work. 
The author was partially supported by JSPS Core-to-core program ``Foundation of a Global Research Cooperative Center in Mathematics focused on Number Theory and Geometry", and KAKENHI 26247004.
　
\subsection{Notation}
\begin{itemize}
\item Let $K$ be a $p$-adic field with the ring of integers $V$ and residue field $k$. Fix a prime element $\pi$ of $V$. Let $W$ be the ring of Witt vectors of $k$, $K_0$ be the fraction field of $W$, and $\sigma$ be the Frobenius automorphism on $K_0$. Let $v_p$ be the additive valuation on $K_0$ normalized by $v_p(p)=1$.
\item For a field $F$, we denote the category of finite dimensional $F$-vector spaces by $\Vec_F$.
\item For an abelian category $\mathcal{A}$, we denote the category of bounded complexes of objects in $\mathcal{A}$ by $C^b(\mathcal{A})$, and its derived category by $D^b(\mathcal{A})$.
\item Assume $\mathcal{C}$ is an additive category with internal $\mathrm{Hom}$. For complexes $L^{\bullet}$ and $M^{\bullet}$ of objects in $\mathcal{C}$, we define a complex $\cHom(L^{\bullet},M^{\bullet})$ by
\[\iHom^i(L^{\bullet},M^{\bullet})=\prod_{j\in\cZ}\iHom(L^j,M^{i+j})\]
with differential given by
\[d^i(f)=(f_{j+1}\circ d_L^j+(-1)^{i+1}d_M^{i+j}\circ f_j)_j\]
for any $f=(f_j)_j\in\prod_{j\in\cZ}\iHom(L^j,M^{i+j})$.
\end{itemize}

\section{Calculation of extension groups}
In this section, we will show that the derived caterory of admissible filtered $(\phi,N)$-modules is equivalent to a certain full subcategory of the derived category of $p$-adic Hodge complexes, by calculating their extension groups.
\subsection{Admissible filtered $(\phi,N)$-modules}
We recall about admissible filtered $(\phi,N)$-modules and their tannakian category introduced by Fontaine \cite{Fontaine}.
\begin{definition}[filtered $(\phi,N)$-module]
A filtered $(\phi,N)$-module over $K$ is a finite dimensional $K_0$-vector space $M$ equipped with additional structures as follows:
\begin{enumerate}
\item A $\sigma$-semilinear isomorphism $\phi$ on $M$, which is called the Frobenius operator.
\item A $K_0$-linear endomorphism $N$ on $M$, which is called the monodromy operator.
\item A decreasing, separated, and exhaustive filtration $F^{\bullet}$ on $M_K=M\otimes_{K_0}K$ by $K$-subspaces, which is called the Hodge filtration.
\item $\phi$ and $N$ satisfy $N\phi=p\phi N$.
\end{enumerate}
A morphism of filtered $(\phi,N)$-modules is a $K_0$-linear map commuting with $\phi$ and $N$ such that the induced $K$-linear map preserve Hodge filtration.
\end{definition}

\begin{remark}
The condition $N\phi=p\phi N$ implies that $N$ is nilpotent.
\end{remark}

Let $M$ be a filtered $(\phi,N)$-module of dimension $d$. Choose a $K_0$-basis $\{e_1,\ldots,e_d\}$ of $M$. Then we can write $\phi(e_i)=\sum_{j=1}^da_{ij}e_j$. Write $A=(a_{ij})_{1\leq i,j\leq d}$. We define the Newton number of $M$ by
\[t_N(M)=v_p(\det A).\]
Given another choice of basis and $A'$, we can write $A'=\sigma(P)AP^{-1}$ with some regular matrix $P$. So $t_N(M)$ is independent of the choice of a basis.

We define the Hodge number of $M$ by
\[t_H(M)=\sum_{n\in\cZ}n\cdot\mathrm{dim}_K\mathrm{Gr}_F^nM_K\]
where $\mathrm{Gr}_F^nM_K=F^nM_K/F^{n+1}M_K$.

\begin{definition}[admissible filtered $(\phi,N)$-module]\label{defadm}
A filtered $(\phi,N)$-module $M$ is called admissible if
\begin{enumerate}
\item $t_H(M)=t_N(M)$,
\item For any filtered $(\phi,N)$-submodule $M'$ of $M$, $t_H(M')\leq t_N(M')$.
\end{enumerate}
Let $\MF$ be the category of admissible filtered $(\phi,N)$-modules over $K$.
\end{definition}

\begin{remark}
In classical terminology, the condition "admissible" in Definition \ref{defadm} is called "weakly admissible", and admissible means "coming from a semistable representation" in the sense of Remark \ref{cateqremark} below. Remark \ref{cateqremark} says that weak admissibility and admissibility are equivalent, so we abuse the terminology.
\end{remark}

$\MF$ has a tannakian structure over $\cQ_p$ as follows (cf. \cite{Fontaine2} section 4.3.4).
For filtered $(\phi,N)$-modules $L$ and $M$, we define $\phi$ and $N$ on $L\otimes_{K_0}M$, and $F^{\bullet}$ on $(L\otimes_{K_0}M)_K=L_K\otimes_KM_K$ by
\begin{eqnarray*}
\phi(\ell\otimes m)&=&\phi_L(\ell)\otimes\phi_M(m)\\
N(\ell\otimes m)&=&N_L(\ell)\otimes m+\ell\otimes N_M(m)\\
F^i(L_K\otimes_KM_K)&=&\sum_{j+j'=i}F_L^jL_K\otimes_KF_M^{j'}M_K
\end{eqnarray*}
for any $\ell\in L$ and $m\in M$. Then $L\otimes_{K_0}M$ becomes a filtered $(\phi,N)$-module. If $L$ and $M$ are admissible, then $L\otimes_{K_0}M$ is also admissible.

Next, we define $\phi$ and $N$ on the set of $K_0$-linear maps $\Hom_{K_0}(L,M)$, and $F^{\bullet}$ on $\Hom_{K_0}(L,M)=\Hom_K(L_K,M_K)$ by
\begin{eqnarray*}
\phi(f)(\ell)&=&\phi_M\circ f\circ\phi_L^{-1}(\ell)\\
N(f)(\ell)&=&N_M\circ f(\ell)-f\circ N_L(\ell)\\
F^i\Hom_K(L_K,M_K)&=&\{g:L_K\rightarrow M_K\mid g(F_L^jL_K)\subset F_M^{i+j}M_K\text{ for all }j\in\cZ\}
\end{eqnarray*}
for any $f\in\Hom_{K_0}(M,M')$ and $m\in M$. Then $\Hom_{K_0}(M,M')$ becomes a filtered $(\phi,N)$-module. We denote it by $\iHom(M,M')$. If $M$ and $M'$ are admissible, then $\iHom(M,M')$ is also admissible.

\begin{remark}\label{cateqremark}
There exists an equivalence of tannakian categories between $\MF$ and the category of semistable $\cQ_p$-representations of the absolute Galois group of $K$. It is proved by Colmez and Fontaine in \cite{CF} first.
\end{remark}

\subsection{Extension groups of admissible filtered $(\phi,N)$-modules}
For objects $L^{\bullet}$ and $M^{\bullet}$ in $C^b(\MF)$, we define the complexes $A^{\bullet}(L^{\bullet},M^{\bullet})$, $B^{\bullet}(L^{\bullet},M^{\bullet})$, and $C^{\bullet}(L^{\bullet},M^{\bullet})$, and maps $\varphi=\varphi(L^{\bullet},M^{\bullet}):A^{\bullet}(L^{\bullet},M^{\bullet})\rightarrow B^{\bullet}(L^{\bullet},M^{\bullet})$ and $\psi=\psi(L^{\bullet},M^{\bullet}):B^{\bullet}(L^{\bullet},M^{\bullet})\rightarrow C^{\bullet}(L^{\bullet},M^{\bullet})$ by
\begin{eqnarray*}
A^{\bullet}(L^{\bullet},M^{\bullet})&=&\cHom(L^{\bullet},M^{\bullet})\oplus F^0\cHom(L_K^{\bullet},M_K^{\bullet})\\
B^{\bullet}(L^{\bullet},M^{\bullet})&=&\cHom(L^{\bullet},M^{\bullet})\oplus\cHom(L^{\bullet},M^{\bullet})\oplus\cHom(L_K^{\bullet},M_K^{\bullet})\\
C^{\bullet}(L^{\bullet},M^{\bullet})&=&\cHom(L^{\bullet},M^{\bullet})\\
\varphi(x,y)&=&(N(x),x-\phi(x),y-x)\\
\psi(x,y,z)&=&x-p\phi(x)-N(y).
\end{eqnarray*}
Then we have the double complex
\begin{equation}\label{doublecomplex}
A^{\bullet}(L^{\bullet},M^{\bullet})\xrightarrow[]{\varphi}B^{\bullet}(L^{\bullet},M^{\bullet})\xrightarrow[]{\psi}C^{\bullet}(L^{\bullet},M^{\bullet})
\end{equation}
of vector spaces over $\cQ_p$ considering $A^n(L^{\bullet},M^{\bullet})$ to be the $(n,0)$-component.
Let
\[\varphi'=\varphi'(L^{\bullet},M^{\bullet}):A^{\bullet}(L^{\bullet},M^{\bullet})\rightarrow\cKer\psi\]
be a map induced by $\varphi$, and put
\begin{eqnarray*}
\Gamma^{\bullet}(L^{\bullet},M^{\bullet})&=&\Cone(A^{\bullet}(L^{\bullet},M^{\bullet})[1]\rightarrow\Cone\psi)=\Cone(\Cone\varphi\rightarrow C^{\bullet}(L^{\bullet},M^{\bullet}))\\
\widetilde{\Gamma}^{\bullet}(L^{\bullet},M^{\bullet})&=&\cCoker\varphi'\\
\hat{\Gamma}^{\bullet}(L^{\bullet},M^{\bullet})&=&\cCoker\psi=\cCoker(\Cone\varphi\rightarrow C^{\bullet}(L^{\bullet},M^{\bullet})).
\end{eqnarray*}
Note that $\Gamma^{\bullet}(L^{\bullet},M^{\bullet})[-2]$ is the total complex of \eqref{doublecomplex}. 
Now we have two distinguished triangles
\begin{equation}\label{distinguished}
\xymatrix{
\cKer\varphi'\ar[r]&\Cone\varphi'[-1]\ar[r]\ar@{=}[ld]&\widetilde{\Gamma}^{\bullet}(L^{\bullet},M^{\bullet})[-1]\\
\cKer(\Cone\varphi\rightarrow C^{\bullet}(L^{\bullet},M^{\bullet}))[-1]\ar[r]&\Gamma^{\bullet}(L^{\bullet},M^{\bullet})[-2]\ar[r]&\hat{\Gamma}^{\bullet}(L^{\bullet},M^{\bullet})[-2].
}\end{equation}

\begin{proposition}\label{limitacyclic1}
Let $L^{\bullet}$ and $M^{\bullet}$ be objects in $C^b(\MF)$. Then
\[\varinjlim_{M'^{\bullet}}H^n(\widetilde{\Gamma}^{\bullet}(L^{\bullet},M'^{\bullet}))=0\]
for every integer $n$. Here $M'^{\bullet}$ runs all quasi-isomorphisms $M^{\bullet}\rightarrow M'^{\bullet}$.
\end{proposition}
\begin{proof}
Since $H^n(\widetilde{\Gamma}^{\bullet}(L^{\bullet},M^{\bullet}))=H^0(\widetilde{\Gamma}^{\bullet}(L^{\bullet},M^{\bullet}[-n]))$, we can put $n=0$. It suffices to prove that for any 0-cocycle $\zeta$ of $\widetilde{\Gamma}^{\bullet}(L^{\bullet},M^{\bullet})$ there exists a quasi-isomorphism $f:M^{\bullet}\rightarrow M'^{\bullet}$ such that $f(\zeta)$ is a coboundary of $\widetilde{\Gamma}^{\bullet}(L^{\bullet},M'^{\bullet})$. We define an object $M'^{\bullet}$ of $C^b(\MF)$ explicitly for $\zeta$ as follows. We simply write $d,N,\phi$ the operators on filtered $(\phi,N)$-modules (omit the subscripts). Take $(x,y,z)\in\Ker^0\psi$ representing $\zeta$. Then
\[x-p\phi x\phi^{-1}-Ny+yN=0\]
and there exists $(s,t)\in A^1(L^{\bullet},M^{\bullet})$ such that
\begin{eqnarray*}
xd-dx&=&Ns-sN\\
yd-dy&=&s-\phi y\phi^{-1}\\
zd-dz&=&s_K-t.
\end{eqnarray*}
We put 
\[M'^i=M^i\oplus L^{i+1}\oplus L^i\oplus L^{i+1}\oplus L^i\oplus L^i\oplus L^{i-1}\]
as a vector space over $K_0$ for every integer $i$. Define the differential, monodromy, and Frobenius operators on $M'^{\bullet}$ by sending  $\eta=(m,\ell_1,\ell_2,\ell_3,\ell_4,\ell_5,\ell_6)\in M'^i$ to
\begin{eqnarray*}
d(\eta)&=&(dm+x\ell_1-dx\ell_2+xd\ell_2+y\ell_3-dy\ell_4+yd\ell_4+s\ell_5+ds\ell_6+sd\ell_6,\\
&&-d\ell_1,\ell_1+d\ell_2,-d\ell_3,\ell_3+d\ell_4,d\ell_5,-\ell_5-d\ell_6)\\
N(\eta)&=&(Nm-Nx\ell_2+xN\ell_2-Ny\ell_4+yN\ell_4,N\ell_1,N\ell_2+\ell_5,N\ell_3,N\ell_4,N\ell_5,N\ell_6)\\
\phi(\eta)&=&(\phi m-\phi x\ell_2+\frac{x\phi\ell_2}{p}-\phi y\ell_4+y\phi\ell_4,\frac{\phi\ell_1}p,\frac{\phi\ell_2}p,\phi\ell_3,\phi\ell_4-\phi\ell_5,\phi\ell_5,\phi\ell_6)
\end{eqnarray*}
respectively. We define the Hodge filtration on $M'^i_K$ by defining $F^kM'^i_K$ to be the set consists of all elements $(m+x\ell_2+y_K\ell_4+z\ell_5-t\ell_6,\ell_1,\ell_2,\ell_3,\ell_4,\ell_5,\ell_6)$ with $m,\ell_1,\ell_2,\ell_3,\ell_4,\ell_5$, and $\ell_6$ are elements of $F^kM_K^i,F^{k+1}L_K^{i+1},F^{k+1}L_K^i,F^kL_K^{i+1},F^kL_K^i,F^kL_K^i$, and $F^kL_K^{i-1}$ respectively.
Then we can verify that $M'^{\bullet}$ is a complex of filterd $(\phi,N)$-modules over $K$ by straightforward calculations. The map $f:M^{\bullet}\rightarrow M'^{\bullet}$ is defined by the canonical inclusion.
Let $L'^{\bullet}$ be the cokernel of $f$ (as a complex of vector spaces over $K_0$) with induced operators. Let $\widetilde{L}^{\bullet}=\Cone(L^{\bullet}\xrightarrow[]{\mathrm{id}}L^{\bullet})$.
Then we have two short exact sequences of complexes of filtered $(\phi,N)$-modules
\begin{eqnarray*}
&&0\rightarrow M^{\bullet}\rightarrow M'^{\bullet}\rightarrow L'^{\bullet}\rightarrow 0\\
&&0\rightarrow\widetilde{L}^{\bullet}(1)\oplus\widetilde{L}^{\bullet}\rightarrow L'^{\bullet}\rightarrow\widetilde{L}^{\bullet}[-1]\rightarrow 0.
\end{eqnarray*}
Since $M^{\bullet}$ and $\widetilde{L}^{\bullet}$ are admissible and admissibility is closed under extension, $M'^{\bullet}$ is also admissible. Since $\widetilde{L}^{\bullet}$ is acyclic, so is $L'^{\bullet}$, and $f$ is a quasi-isomorphism. To prove $f(\zeta)$ is a coboundary, we define elements $(a,b,c)\in\Ker^{-1}\psi(L^{\bullet},M'^{\bullet})$ and $(\lambda,\mu)\in A^0(L^{\bullet},M'^{\bullet})$ by
\begin{eqnarray*}
&&a(\ell)=(0,\ell,0,0,0,0,0),~b(\ell)=(0,0,0,\ell,0,0,0),~c(\ell')=(0,0,0,0,0,0,0)\\
&&\lambda(\ell)=(0,0,0,0,0,-\ell,0),~\mu(\ell')=(-z\ell',0,0,0,0,-\ell',0)
\end{eqnarray*}
for every $\ell\in L^i$ and $\ell'\in L^i_K$. Then
\begin{eqnarray*}
fx&=&da+ad+N\lambda-\lambda N\\
fy&=&db+bd+\lambda-\phi\lambda\phi^{-1}\\
fz&=&dc+cd+\lambda_K-\mu
\end{eqnarray*}
so we complete the proof.
\end{proof}

\begin{proposition}\label{limitacyclic2}
Let $L^{\bullet}$ and $M^{\bullet}$ be objects in $C^b(\MF)$. Then
\[\varinjlim_{M'^{\bullet}}H^n(\hat{\Gamma}^{\bullet}(L^{\bullet},M'^{\bullet}))=0\]
for every integer $n$. Here $M'^{\bullet}$ runs all quasi-isomorphisms $M^{\bullet}\rightarrow M'^{\bullet}$.
\end{proposition}

\begin{proof}
We fix an $L^{\bullet}$ in $C^b(\MF)$.
Then $\hat{\Gamma}(L^{\bullet},M^{\bullet})=\Coker^{\bullet}\xi_{M^{\bullet}}$ where
\[\xi=\xi_{M^{\bullet}}:\cHom(L^{\bullet},M^{\bullet})\oplus\cHom(L^{\bullet},M^{\bullet})\rightarrow\cHom(L^{\bullet},M^{\bullet}):(x,y)\mapsto x-p\phi(x)-N(y),\]
By shifting, we may assume $n=0$. It suffices to show there exists a quasi-isomorphism $f:M^{\bullet}\rightarrow M'^{\bullet}$ such that $f(x)$ lies in the image of $\xi_{M'^{\bullet}}$ for any $x\in\iHom^0(L^{\bullet},M^{\bullet})$.
We define such $M'^{\bullet}$ as follows.
Let $\widetilde{M}^{\bullet}=\Cone(M^{\bullet}\xrightarrow[]{\mathrm{id}}M^{\bullet})[-1]$. For $k\geq 1$ we define
\[M^{\bullet}_k=M^{\bullet}\oplus\widetilde{M}^{\bullet}(1)\oplus\cdots\oplus\widetilde{M}^{\bullet}(k)\]
as a complex of filtered vector spaces over $K_0$, and put $M_0^{\bullet}=M^{\bullet}$. We define the $i$-th monodromy operator and Frobenius operator on $M_k^{\bullet}$ by sending $\eta=(y_0,y_1,z_1,\ldots,z_{k-1},y_k,z_k)\in M'^i$ to
\begin{eqnarray*}
N(\eta)&=&(Ny_0-y_1,Ny_1-y_2,Nz_1-z_2,\ldots,Nz_{k-1}-z_k,Ny_k,Nz_k)\\
\phi(\eta)&=&(\phi y_0-p\phi y_1,p\phi y_1-p^2\phi y_2,p\phi z_1-p^2\phi z_2,\ldots,p^{k-1}\phi z_{k-1}-p^k\phi z_k,p^k\phi y_k,p^k\phi z_k)
\end{eqnarray*}
respectively. Then we can verify that $M_k^{\bullet}$ is a complex of filtered Frobenius monodromy modules over $K$ by straightforward calculations. Since
\[0\rightarrow M_{k-1}^{\bullet}\rightarrow M_k^{\bullet}\rightarrow\widetilde{M}^{\bullet}(k)\rightarrow 0\]
is exact and $\widetilde{M}^{\bullet}(k)$ is admissible, each $M_k^{\bullet}$ is also admissible. 
Let $r=2r_0$ be an integer such that $N^{r_0}$ on $M^{\bullet}$ vanish, and put $M'^{\bullet}=M_r^{\bullet}$. Let $f:M^{\bullet}\rightarrow M'^{\bullet}$ be the natural inclusion.
For $x\in\iHom^0(L^{\bullet},M^{\bullet})$, we define $a\in\iHom^0(L^{\bullet},M'^{\bullet})$ by
\begin{eqnarray*}&&a(\ell)=(0,x\ell,0,Nx\ell-xN\ell,0,N^2x\ell-2NxN\ell+xN^2\ell,0,\\
&&\hspace{90pt}\ldots,\sum_{0\leq k\leq r-1}(-1)^k\left(\begin{array}{c}r-1\\ k\end{array}\right)N^{r-1-k}xN^k\ell,0).\end{eqnarray*}
Then we have
\[Na(\ell)-aN(\ell)=(-x\ell,0,0,\ldots,0)\]
and
\[f(x)=\xi_{M'^{\bullet}}(0,a).\]
\end{proof}

\begin{lemma}\label{homotopycathom}
Let $L^{\bullet}$ and $M^{\bullet}$ be two objects in $C^b(\MF)$. We denote $\varphi=\varphi(L^{\bullet},M^{\bullet})$. Then we have the canonical isomorphism
\[H^n(\cKer\varphi)\cong\Hom_{K^b(\MF)}(L^{\bullet},M^{\bullet}[n]).\]
\end{lemma}

\begin{proof}
Consider the condition for that an element $\zeta=(x_j,y_j)_j\in A^n(L^{\bullet},M^{\bullet})$ defines a map $L^{\bullet}\rightarrow M^{\bullet}[n]$ in $C^b(\MF)$. Then we notice that preservability of Hodge filtration and compatibility of monodromy operator and Frobenius operator are equivalent to the condition $\zeta$ is in the kernel of $\varphi$. Compatibility of differential means $\zeta$ is an $n$-cocycle of $\cKer\varphi$. So the group of $n$-cocycles of $\cKer\varphi$ is isomorphic to $\Hom_{C^b(\MF)}(L^{\bullet},M^{\bullet}[n])$. Finally, $\zeta$ is a coboundary if and only if the map $L^{\bullet}\rightarrow M^{\bullet}[n]$ corresponding to $\zeta$ is homotopic to zero map.
\end{proof}

\begin{lemma}\label{coneqis}
Let $L^{\bullet}$ be an object in $C^b(\MF)$. For every quasi-isomorphism $M^{\bullet}\rightarrow M'^{\bullet}$ in $C^b(\MF)$, the induced map $\Gamma^{\bullet}(L^{\bullet},M^{\bullet})\rightarrow\Gamma^{\bullet}(L^{\bullet},M'^{\bullet})$ in $C^b(\Vec_{\cQ_p})$ is also quasi-isomorphic.
\end{lemma}

\begin{proof}
Since internal $\Hom$ and taking cone define exact functors, $\Gamma^{\bullet}(L^{\bullet},-)$ is also exact. So the lemma follows.
\end{proof}

\begin{theorem}\label{extmf}
For every objects $L^{\bullet}$ and $M^{\bullet}$ of $C^b(\MF)$, we have the canonical isomorphism
\[\Ext^n(L^{\bullet},M^{\bullet})=\Hom_{D^b(\MF)}(L^{\bullet},M^{\bullet}[n])\cong H^n(\Gamma^{\bullet}(L^{\bullet},M^{\bullet})[-2]).\]
\end{theorem}

\begin{proof}
Taking direct limit of the cohomology long exact sequence obtained from distinguished triangles \eqref{distinguished} for all quasi-isomorphisms $M^{\bullet}\rightarrow M'^{\bullet}$, by Proposition \ref{limitacyclic1} and Proposition \ref{limitacyclic2}, we have the canonical isomorphism
\begin{equation}\label{extmfeq}
\varinjlim_{M'^{\bullet}}H^n(\cKer\varphi'(L^{\bullet},M'^{\bullet}))\cong\varinjlim_{M'^{\bullet}}H^n(\Gamma^{\bullet}(L^{\bullet},M'^{\bullet})[-2]).
\end{equation}
By Lemma \ref{homotopycathom} and $\cKer\varphi=\cKer\varphi'$, the left hand side in \eqref{extmfeq} is canonically isomorphic to
\[\varinjlim_{M'^{\bullet}}\Hom_{K^b(\MF)}(L^{\bullet},M'^{\bullet}[n])=\Hom_{D^b(\MF)}(L^{\bullet},M^{\bullet}[n]).\]
By  Lemma \ref{coneqis}, the right hand side in \eqref{extmfeq} is canonically isomorphic to $H^n(\Gamma^{\bullet}(L^{\bullet},M^{\bullet})[-2])$.
\end{proof}

\subsection{$p$-adic Hodge complexes}
We define the category of $p$-adic Hodge complexes and its derived category following the argument of \cite{Bannai}.

\begin{definition}
Let $C^b_{\rig,K_0}$ be the category of triples $(M^{\bullet},\phi,N)$ where
\begin{enumerate}
\setlength{\parskip}{0pt}
\setlength{\itemsep}{0pt}
\item $M^{\bullet}$ is an object in $C^b(\Vec_{K_0})$.
\item $\phi:M^{\bullet}\rightarrow M^{\bullet}$ is a $\sigma$-semilinear endomorphism of complexes called the Frobenius operator.
\item $N:M^{\bullet}\rightarrow M^{\bullet}$ is a nilpotent endomorphism in $C^b(\Vec_{K_0})$ called the monodromy operator.
\item $\phi$ and $N$ satisfies $N\phi=p\phi N$.
\end{enumerate}
A morphism in $C^b_{\rig,K_0}$ is a morphism in $C^b(\Vec_{K_0})$ which is compatible with $\phi$ and $N$.
\end{definition}

\begin{definition}
Let $C^b_{\dR,K}$ be the category of pairs $(M^{\bullet},F)$ where
\begin{enumerate}
\setlength{\parskip}{0pt}
\setlength{\itemsep}{0pt}
\item $M^{\bullet}$ is an object in $C^b(\Vec_K)$.
\item $F$ is a separated exhaustive descending filtration on $M^{\bullet}$ called the Hodge filtration.
\end{enumerate}
A morphism in $C^b_{\dR,K}$ is a morphism in $C^b(\Vec_K)$ which preserves $F$.
\end{definition}

\begin{definition}
A $p$-adic Hodge complex is a system $M^{\bullet}=(M_{\rig}^{\bullet},\phi,N,M_K^{\bullet},M_{\dR}^{\bullet},F,\alpha,\beta)$ where
\begin{enumerate}
\setlength{\parskip}{0pt}
\setlength{\itemsep}{0pt}
\item $(M_{\rig}^{\bullet},\phi,N)$ is an object in $C^b_{\rig,K_0}$.
\item $M_K^{\bullet}$ is an object in $C^b(\Vec_K)$.
\item $(M_{\dR},F)$ is an object in $C^b_{\dR,K}$.
\item $\alpha:M_{\rig}^{\bullet}\otimes_{K_0}K\rightarrow M_K^{\bullet}$ and $\beta:M_{\dR}^{\bullet}\rightarrow M_K^{\bullet}$ are morphisms in $C^b(\Vec_K)$. We call them comparision maps.
\end{enumerate}
We call $M_{\rig}^{\bullet}$ (resp. $M^{\bullet}_K$, $M^{\bullet}_{\dR}$) the rigid (resp. $K$-, de Rham) specialization of $M^{\bullet}$.
A morphism of $p$-adic Hodge complexes is a triple of morphisms on specializations which are compatible with comparison maps. We denote $\pHC$ the category of $p$-adic Hodge complexes.
\end{definition}

A morphism in $\pHC$ is homotopic to zero if it has a homotopy on every specialization. We define the homotopy category $\pHK$ of $p$-adic Hodge complexes to be $\pHC$ modulo morphisms homotopic to zero. A $p$-adic Hodge complex is acyclic if every specialization is acyclic. Then the localization of $\pHK$ by acyclic objects exists, and has the natural structure of triangulated category. We denote it by $\pHD$.

\subsection{Extension groups of $p$-adic Hodge complexes}
Note that $C^b_{\rig,K_0}$ and $C^b_{\dR,K}$ are additive categories with internal $\mathrm{Hom}$. Let $L^{\bullet}$ and $M^{\bullet}$ be objects in $\pHC$. We assume that the Frobenius operator $\phi$ on $L^{\bullet}$ is an automorphism. Then we define the complexes $\mathcal{A}^{\bullet}(L^{\bullet},M^{\bullet})$, $\mathcal{B}^{\bullet}(L^{\bullet},M^{\bullet})$, and $\mathcal{C}^{\bullet}(L^{\bullet},M^{\bullet})$ and maps $\Phi=\Phi(L^{\bullet},M^{\bullet}):\mathcal{A}^{\bullet}(L^{\bullet},M^{\bullet})\rightarrow\mathcal{B}^{\bullet}(L^{\bullet},M^{\bullet})$ and $\Psi=\Psi(L^{\bullet},M^{\bullet}):\mathcal{B}^{\bullet}(L^{\bullet},M^{\bullet})\rightarrow\mathcal{C}^{\bullet}(L^{\bullet},M^{\bullet})$ by
\begin{eqnarray*}
\mathcal{A}^{\bullet}(L^{\bullet},M^{\bullet})&=&\cHom(L_{\rig}^{\bullet},M_{\rig}^{\bullet})\oplus\cHom(L_K^{\bullet},M_K^{\bullet})\oplus F^0\cHom(L_{\dR}^{\bullet},M_{\dR}^{\bullet})\\
\mathcal{B}^{\bullet}(L^{\bullet},M^{\bullet})&=&\cHom(L_{\rig}^{\bullet},M_{\rig}^{\bullet})\oplus\cHom(L_{\rig}^{\bullet},M_{\rig}^{\bullet})\oplus\cHom(L_{\rig}^{\bullet},M_K^{\bullet})\oplus\cHom(L_{\dR}^{\bullet},M_K^{\bullet})\\
\mathcal{C}^{\bullet}(L^{\bullet},M^{\bullet})&=&\cHom(L_{\rig}^{\bullet},M_{\rig}^{\bullet})\\
\Phi(x,y,z)&=&(N(x),x-\phi(x),\alpha x-y\alpha,y\beta-\beta z)\\
\Psi(x,y,z,w)&=&x-p\phi(x)-N(y).
\end{eqnarray*}

Then we have a double complex
\begin{equation}\label{doublecomplex2}
\mathcal{A}^{\bullet}(L^{\bullet},M^{\bullet})\xrightarrow[]{\Phi}\mathcal{B}^{\bullet}(L^{\bullet},M^{\bullet})\xrightarrow[]{\Psi}\mathcal{C}^{\bullet}(L^{\bullet},M^{\bullet})
\end{equation}
of vector spaces over $\cQ_p$ considering $\mathcal{A}^n(L^{\bullet},M^{\bullet})$ to be the $(n,0)$-component.
Let
\[\Phi'=\Phi'(L^{\bullet},M^{\bullet}):\mathcal{A}^{\bullet}(L^{\bullet},M^{\bullet})\rightarrow\cKer\Psi\]
be a map induced by $\Phi$, and put
\begin{eqnarray*}
\Lambda^{\bullet}(L^{\bullet},M^{\bullet})&=&\Cone(\mathcal{A}^{\bullet}(L^{\bullet},M^{\bullet})[1]\rightarrow\Cone\Psi)=\Cone(\Cone\Phi\rightarrow\mathcal{C}^{\bullet}(L^{\bullet},M^{\bullet}))\\
\widetilde{\Lambda}^{\bullet}(L^{\bullet},M^{\bullet})&=&\cCoker\Phi'\\
\hat{\Lambda}^{\bullet}(L^{\bullet},M^{\bullet})&=&\cCoker\Psi=\cCoker(\Cone\Phi\rightarrow\mathcal{C}^{\bullet}(L^{\bullet},M^{\bullet})).
\end{eqnarray*}
Note that $\Lambda^{\bullet}(L^{\bullet},M^{\bullet})$ is the total complex of \eqref{doublecomplex2}. 
Now we have two distinguished triangles
\begin{equation}\label{distinguished2}
\xymatrix{
\cKer\Phi'\ar[r]&\Cone\Phi'[-1]\ar[r]\ar@{=}[ld]&\widetilde{\Lambda}^{\bullet}(L^{\bullet},M^{\bullet})[-1]\\
\cKer(\Cone\Phi\rightarrow\mathcal{C}^{\bullet}(L^{\bullet},M^{\bullet}))[-1]\ar[r]&\Lambda^{\bullet}(L^{\bullet},M^{\bullet})[-2]\ar[r]&\hat{\Lambda}^{\bullet}(L^{\bullet},M^{\bullet})[-2].
}\end{equation}

\begin{remark}\label{coneremark}
For an object $M^{\bullet}$ in $\pHC$, we define the complexes $\mathcal{A}_0^{\bullet}(M^{\bullet})$, $\mathcal{B}_0^{\bullet}(M^{\bullet})$, $\mathcal{C}_0^{\bullet}(M^{\bullet})$, and $\Lambda_0^{\bullet}(M^{\bullet})$ by
\begin{eqnarray*}
\mathcal{A}_0^{\bullet}(M^{\bullet})&=&M_{\rig}^{\bullet}\oplus F^0M_{\dR}^{\bullet}\\
\mathcal{B}_0^{\bullet}(M^{\bullet})&=&M_{\rig}^{\bullet}\oplus M_{\rig}^{\bullet}\oplus M_K^{\bullet}\\
\mathcal{C}_0^{\bullet}(M^{\bullet})&=&M_{\rig}^{\bullet}\\
\Lambda_0^{\bullet}(M^{\bullet})&=&\Cone(\mathcal{A}_0^{\bullet}(M^{\bullet})[1]\rightarrow\Cone\Psi_0)=\Cone(\Cone\Psi_0\rightarrow\mathcal{C}_0^{\bullet}(M^{\bullet})).
\end{eqnarray*}
where
\begin{eqnarray*}
&&\Phi_0=\Phi_0(M^{\bullet}):\mathcal{A}_0^{\bullet}(M^{\bullet})\rightarrow\mathcal{B}_0^{\bullet}(M^{\bullet}):(x,y)\mapsto(N(x),x-\phi(x),\alpha(x)-\beta(y))\\
&&\Psi_0=\Psi_0(M^{\bullet}):\mathcal{B}_0^{\bullet}(M^{\bullet})\rightarrow\mathcal{C}_0^{\bullet}(M^{\bullet}):(x,y,z)\mapsto x-p\phi(x)-N(y).
\end{eqnarray*}
Then the maps
\begin{eqnarray*}
\mathcal{A}^{\bullet}(K_0,M^{\bullet})\rightarrow\mathcal{A}_0^{\bullet}(M^{\bullet})&:&(x,y,z)\mapsto(x,z)\\
\mathcal{B}^{\bullet}(K_0,M^{\bullet})\rightarrow\mathcal{B}_0^{\bullet}(M^{\bullet})&:&(x,y,z,w)\mapsto(x,y,z+w)\\
\mathcal{C}^{\bullet}(K_0,M^{\bullet})\rightarrow\mathcal{C}_0^{\bullet}(M^{\bullet})&:&x\mapsto x
\end{eqnarray*}
induce the quasi-isomorphism
\[\Lambda^{\bullet}(K_0,M^{\bullet})\rightarrow\Lambda_0^{\bullet}(M^{\bullet}).\]
\end{remark}

\begin{proposition}\label{limitacyclic3}
Let $L^{\bullet}$ and $M^{\bullet}$ be objects in $\pHC$. Assume that comparision maps of $L^{\bullet}$ are identity maps, and $\phi$ on $L^{\bullet}_{\rig}$ is an automorphism. Then
\[\varinjlim_{M'^{\bullet}}H^n(\widetilde{\Lambda}^{\bullet}(L^{\bullet},M'^{\bullet}))=0\]
for every integer $n$. Here $M'^{\bullet}$ runs all quasi-isomorphisms $M^{\bullet}\rightarrow M'^{\bullet}$.
\end{proposition}

\begin{proof}
We can prove this by the idea similar to Proposition \ref{limitacyclic1}. We assume $n=0$. Denote the compression maps of $M^{\bullet}$ by $\alpha$ and $\beta$. Let $\zeta$ be a $0$-cocycle of $\widetilde{\Lambda}(L^{\bullet},M^{\bullet})$. Take $(x,y,z,w)\in\Ker^0\Psi$ representing $\zeta$. Then there exists $(s,t,u)\in\mathcal{A}^1(L^{\bullet},M^{\bullet})$ such that\\
\[-dx=Ns,~-dy=s-\phi s,~-dz=\alpha s_K-t,~-dw=t-\beta u.\]
Let
\begin{eqnarray*}
M'^i_{\rig}&=&M^i_{\rig}\oplus L_{\rig}^{i+1}\oplus L_{\rig}^i\oplus L_{\rig}^{i+1}\oplus L_{\rig}^i\oplus L_{\rig}^i\oplus L_{\rig}^{i-1}\\
M'^i_K&=&M_K^i\oplus L_K^{i+1}\oplus L_K^i\oplus L_K^{i+1}\oplus L_K^i\oplus L_K^i\oplus L_K^{i-1}\\
M^i_{\dR}&=&M^i_{\dR}\oplus L_{\dR}^i\oplus L_{\dR}^{i-1}.
\end{eqnarray*}
We define the differential, monodromy, and Frobenius operators on $M'^{\bullet}_{\rig}$ same as proof of Proposition \ref{limitacyclic1}. Define the differentials on $M'^{\bullet}_K$ and $M'^{\bullet}_{\dR}$ by
\begin{eqnarray*}
&&d(m,\ell_1,\ell_2,\ell_3,\ell_4,\ell_5,\ell_6)\\
&&\hspace{10pt}=(dm+\alpha x\ell_1-\alpha dx\ell_2+\alpha xd\ell_2+\alpha y_K\ell_3-\alpha dy_K\ell_4-\alpha y_Kd\ell_4+t\ell_5+dt\ell_6+td\ell_6,\\
&&\hspace{25pt}-d\ell_1,\ell_1+d\ell_2,-d\ell_3,\ell_3+d\ell_4,d\ell_5,-\ell_5-d\ell_6)\\
&&d(m,\ell_1,\ell_2,\ell_3,\ell_4,\ell_5,\ell_6)\\
&&\hspace{10pt}=(dm+u\ell_1+du\ell_2+ud\ell_2,d\ell_1,-\ell_1-d\ell_2)
\end{eqnarray*}
respectively. The Hodge filtration on $M'^{\bullet}_{\dR}$ is defined by direct sum.
The comparision maps $M'^{\bullet}_{\rig}\otimes K\rightarrow M'^{\bullet}_K$ and $M'^{\bullet}_{\dR}\rightarrow M'^{\bullet}_K$ are defined by
\begin{eqnarray*}
(m,\ell_1,\ell_2,\ell_3,\ell_4,\ell_5,\ell_6)&\mapsto&(\alpha m-z_K\ell_5-dz_K\ell_6+z_Kd\ell_6,\ell_1,\ell_2,\ell_3,\ell_4,\ell_5,\ell_6)\\
(m,\ell_1,\ell_2)&\mapsto&(\beta m+w\ell_1+dw\ell_2+wd\ell_2,0,0,0,0,\ell_1,\ell_2).
\end{eqnarray*}
Define $f=(f_{\rig},f_K,f_{\dR}):M^{\bullet}\rightarrow M'^{\bullet}$ as inclusions to the first summands on specializations.
We define elements $(a,b,c,e)\in\Ker^{-1}\Psi(L^{\bullet},M'^{\bullet})$ and $(\lambda,\mu,\nu)\in\mathcal{A}^0(L^{\bullet},M'^{\bullet})$ by $c=0,~e=0$, and
\begin{eqnarray*}
&&a(\ell)=(0,\ell,0,0,0,0,0),~b(\ell)=(0,0,0,\ell,0,0,0)\\
&&\lambda(\ell)=(0,0,0,0,0,-\ell,0),~\mu(\ell')=(0,0,0,0,0,-\ell',0),~\nu(\ell')=(0,-\ell',0)
\end{eqnarray*}
for every $\ell\in L_{\rig}^i$ and $\ell'\in L_K^i=L_{\dR}^i$. Then we have
\begin{eqnarray*}
f_{\rig}x=da+ad+N\lambda-\lambda N\\
f_{\rig}y=db+bd+\lambda-\phi\lambda\phi^{-1}\\
f_Kz=dc+cd+\alpha\lambda-\mu\alpha\\
f_{\dR}w=de+ed+\mu\beta-\beta\nu
\end{eqnarray*}
so we complete the proof.
\end{proof}

\begin{proposition}\label{limitacyclic4}
Let $L^{\bullet}$ and $M^{\bullet}$ be objects in $\pHC$. Assume that $\phi$ on $L^{\bullet}_{\rig}$ is an automorphism. Then
\[\varinjlim_{M'^{\bullet}}H^n(\hat{\Lambda}^{\bullet}(L^{\bullet},M'^{\bullet}))=0\]
for every integer $n$. Here $M'^{\bullet}$ runs all quasi-isomorphisms $M^{\bullet}\rightarrow M'^{\bullet}$.
\end{proposition}

\begin{proof}
Define $\xi_{M^{\bullet}}:\cHom(L_{\rig}^{\bullet},M_{\rig}^{\bullet})\oplus\cHom(L_{\rig}^{\bullet},M_{\rig}^{\bullet})\rightarrow \cHom(L_{\rig}^{\bullet},M_{\rig}^{\bullet})$ by $\xi(x,y)=x-p\phi(x)-N(y)$. 
Let $r=2r_0$ be an integer such that $N^{r_0}$ on $M^{\bullet}$ is zero. Let $\widetilde{M}^{\bullet}=\Cone(M^{\bullet}\xrightarrow[]{\mathrm{id}}M^{\bullet})[-1]$. For $?\in\{\rig,K,\dR\}$, put $M'^{\bullet}_?=M^{\bullet}_?\oplus \widetilde{M}^{\bullet}_?(1)\oplus\cdots\oplus\widetilde{M}^{\bullet}_?(r)$ as complexes. Define the monodromy and Frobenius operators on $M'^{\bullet}_{\rig}$ as the proof of Proposition \ref{limitacyclic2}. Let $f:M^{\bullet}\rightarrow M'^{\bullet}$ the natural inclusion. Then $f(x)$ lies in the image of $\xi_{M'^{\bullet}}$ for any $x\in\iHom^0(L^{\bullet}_{\rig},M^{\bullet}_{\rig})$.
\end{proof}

\begin{theorem}\label{exthdg}
Let $L^{\bullet}$ and $M^{\bullet}$ be objects in $\pHC$. Assume that comparision maps of $L^{\bullet}$ are identity maps, and $\phi$ on $L_{\rig}^{\bullet}$ is an automorphism. Then we have the canonical isomorphism
\[\Ext^n(L^{\bullet},M^{\bullet})=\Hom_{\pHD}(L^{\bullet},M^{\bullet}[n])\cong H^n(\Lambda^{\bullet}(L^{\bullet},M^{\bullet})[-2]).\]
\end{theorem}

\begin{proof}
It follows from Proposition \ref{limitacyclic3} and Proposition \ref{limitacyclic4} in the same way as Theorem \ref{extmf}.
\end{proof}

\subsection{Equivalence of derived categories}
Let $\pHD^{\leq 0}$ (resp. $\pHD^{\geq 0}$) be the full subcategory of $\pHD$ consisting of objects such that the specializations are acyclic in degree $>0$ (resp. $<0$). Then $(\pHD^{\leq 0},\pHD^{\geq 0})$ is a non-degenerate $t$-structure on $\pHD$. The natural functor $C^b(\MF)\rightarrow \pHC$ induces the functor $\Theta:D^b(\MF)\rightarrow \pHD$ which is compatible with $t$-structures. Through this functor, we regard $\MF$ to be a subcategory of the heart of the $t$-structure on $\pHD$.

We say that an object in $\pHD$ is strict if it is represented by an object $M^{\bullet}$ in $\pHC$ such that $d^i(F^jM^i_{\dR})=d^i(M^i_{\dR})\cap F^jM^{i+1}_{\dR}$ for any $i$ and $j$. Let $\widetilde{\pHD}$ be the full subcategory of strict objects in $\pHD$ whose cohomology objects with respect to the $t$-structure are in $\MF$.
Then $\widetilde{\pHD}$ is a trianglated subcategory of $\pHD$, and the heart of induced $t$-structure is $\MF$.

\begin{theorem}\label{cateqthm}
The functor $\Theta:D^b(\MF)\rightarrow\widetilde{\pHD}$ is an equivalence.
\end{theorem}

\begin{proof}
By definition, this functor induces an equivalence on the hearts. Thus, by \cite{Beilinson} Lemma 1.4, it suffices to show that for any objects $L$ and $M$ in $\MF$ and $n>0$ $\Ext^n(L,M)\rightarrow\Ext^n(\Theta(L),\Theta(M))$ is an isomorphism. By Theorem \ref{extmf} and Theorem \ref{exthdg}, it suffices to show that $\Gamma^{\bullet}(L,M)[-2]$ and $\Lambda^{\bullet}(\Theta(L),\Theta(M))[-2]$ are quasi-isomorphic. Explicitly, $\Gamma^{\bullet}(L,M)[-2]$ is
\begin{eqnarray*}
&&\underline{\Hom}(L,M)\oplus F^0\underline{\Hom}(L_K,M_K)\\
&\rightarrow&\underline{\Hom}(L,M)\oplus\underline{\Hom}(L,M)\oplus\underline{\Hom}(L_K,M_K)\\
&\rightarrow&\underline{\Hom}(L,M)
\end{eqnarray*}
here two arrows are defined by $(x,y)\mapsto(N(x),x-\phi(x),y-x)$ and $(x,y,z)\mapsto x-p\phi(x)-N(y)$.
And $\Lambda^{\bullet}(\Theta(L),\Theta(M))[-2]$ is
\begin{eqnarray*}
&&\underline{\Hom}(L,M)\oplus\underline{\Hom}(L_K,M_K)\oplus F^0\underline{\Hom}(L_K,M_K)\\
&\rightarrow&\underline{\Hom}(L,M)\oplus\underline{\Hom}(L,M)\oplus\underline{\Hom}(L,M_K)\oplus\underline{\Hom}(L_K,M_K)\\
&\rightarrow&\underline{\Hom}(L,M),
\end{eqnarray*}
here two arrows are defined by $(x,y,z)\mapsto(N(x),x-\phi(x),x-y,y-z)$ and $(x,y,z,w)\mapsto x-p\phi(x)-N(y)$.
One can easily check that these are quasi-isomorphic.
\end{proof}

\section{Log rigid cohomology}
\subsection{Generalized Godement resolution}
To define the functorial $p$-adic Hodge complexes associated to strictly semistable schemes, we need a functorial flasque resolution of sheaves on dagger spaces. For this, dagger spaces do not have enough points. Namely, a sheaf on a dagger space can be non-trivial even if its stalk at every point of underlying set is trivial. Using the points in the sense of van der Put and Schneider \cite{vdPS}, we obtain the generalized Godement resolution following \cite{CCM}.

For a site $X$, we write the category of sheaves of abelian groups on $X$ by $\Sh(X)$.

\begin{definition}
Let $u:P\rightarrow X$ be a morphism of sites, $u^*:\Sh(X)\rightarrow\Sh(P)$ be the functor induced by $u$. Let $\eta:\mathrm{id}_{\Sh(X)}\rightarrow u_*u^*$ and $\epsilon:u^*u_*\rightarrow\mathrm{id}_{\Sh(P)}$ be the natural transformations given by adjoint property. For an object $\mathcal{F}$ in $\Sh(X)$ and $n\geq 0$, let $B^{n+1}(\mathcal{F})=(u_*u^*)^n(\mathcal{F})$. Then we get a co-simplicial sheaf $B^{\bullet}(\mathcal{F})$ with $(u_*u^*)^iu_*\epsilon u^*(u_*u^*)^{n-1-i}$ and $(u_*u^*)^i\eta(u_*u^*)^{n-i}$ as $i$-th co-degeneracy and $i$-th co-face. Let $\Gd_P\mathcal{F}$ be the associated complex of $\Sh(X)$.

For a complex $\mathcal{F}^{\bullet}$ of sheaves on $X$, let $\Gd_P(\mathcal{F}^{\bullet})$ be the total complex of the double complex $\Gd_P(\mathcal{F}^i)^j$.
\end{definition}

\begin{proposition}[\cite{CCM} Lemma 3.1]\label{godement1}
There is a canonical morphism $b_{\mathcal{F}}:\mathcal{F}\rightarrow\Gd_P\mathcal{F}$, which is quasi-isomorphism if $u^*$ is exact and conservative. 
\end{proposition}

\begin{proposition}[\cite{CCM} Lemma 3.2]\label{godement2}
Suppose we are given a commutative diagram of sites
\[\xymatrix{
P\ar[r]^g\ar[d]_u\comm[rd]&Q\ar[d]^v\\
X\ar[r]_f&Y,
}\]
and sheaves $\mathcal{F}$ on $Y$ and $\mathcal{G}$ on $X$, and a morphism $a:\mathcal{G}\rightarrow f_*\mathcal{F}$. Then there exists a canonical morphism $\Gd_Q\mathcal{G}\rightarrow f_*\Gd_P\mathcal{F}$ which is compatible with $b_{\mathcal{F}}$ and $b_{\mathcal{G}}$.
\end{proposition}

\begin{definition}[prime filter]
Let $\mathcal{X}$ be a rigid analytic space over $K$. A prime filter on $\mathcal{X}$ is a system $\mathfrak{p}$ of admissible open subsets of $\mathcal{X}$ such that
\begin{itemize}
\item $\mathfrak{p}$ contains $\emptyset$ and $\mathcal{X}$.
\item If $\mathcal{U}_1,\mathcal{U}_2\in\mathfrak{p}$, then $\mathcal{U}_1\cap\mathcal{U}_2\in\mathfrak{p}$.
\item If $\mathcal{U}\in\mathfrak{p}$, then every admissible open subset of $\mathcal{X}$ which contains $\mathcal{U}$ is also in $\mathfrak{p}$.
\item If $\mathcal{U}\in\mathfrak{p}$ and $\{\mathcal{U}_i\}_{i\in I}$ is an admissible covering of $\mathcal{U}$, then $\mathcal{U}_{i_0}\in\mathfrak{p}$ for some $i_0\in I$.
\end{itemize}
\end{definition}

For an admissible open subset $\mathcal{U}$ of $\mathcal{X}$, Let $\widetilde{\mathcal{U}}$ be the set of prime filters on $\mathcal{X}$ containing $\mathcal{U}$. 
Let $P'(\mathcal{X})$ (resp. $Pt'(\mathcal{X})$) be the sets of prime filters on $\mathcal{X}$ with the topology generated by all $\widetilde{\mathcal{U}}$'s (resp. with the discrete topology). Let $P(\mathcal{X})$ and $Pt(\mathcal{X})$ be the site associated to topological spaces $P'(\mathcal{X})$ and $Pt'(\mathcal{X})$. Let $\sigma:P(\mathcal{X})\rightarrow\mathcal{X}$ be the morphism of sites defined by $\mathcal{U}\mapsto\widetilde{\mathcal{U}}$. Let $\xi:Pt(\mathcal{X})\rightarrow\mathcal{X}$ be the composition of the natural morphism $Pt(\mathcal{X})\rightarrow P(\mathcal{X})$ with $\sigma$.

\begin{proposition}[\cite{CCM} Lemma 3.8]\label{godement3}
$\xi^*:\Sh(\mathcal{X})\rightarrow\Sh(Pt(\mathcal{X}))$ is exact and conservative.
\end{proposition}

Since every dagger space $\mathcal{Y}$ is homeomorphic to its completion $\hat{\mathcal{Y}}$ as Grothendieck topological spaces, by Proposition \ref{godement1} and Proposition \ref{godement3} we can define the Godement resolution $\Gd_{\an}\mathcal{F}$ of $\mathcal{F}\in\Sh(\mathcal{Y})$ as $\Gd_{Pt(\hat{\mathcal{Y}})}\mathcal{F}$.

\subsection{Log rigid complexes}
Contents of this subsection are almost quotations from the first section of \cite{Grosse1} with some arranging notations. In the following, we use generalized Godement resolution $\Gd_{\an}$ to define log rigid complexes. We refer \cite{Me}, \cite{Grosse0} and \cite{LM} for weak formal schemes and dagger spaces.

Let $\widetilde{S}$ be the log scheme $\Spec W[t]$ with the log structure associated to $\cN\rightarrow W[t]:1\mapsto t$, $S$ be the exact closed log subscheme of $\widetilde{S}$ defined by $(p)$, $S_0$ be the exact closed log subscheme of $S$ defined by $(t)$. Let $\mathfrak{S}$ be the weak completion of $\widetilde{S}$. Let $\mathfrak{S}_W$ and $\mathfrak{S}_V$ be the exact closed weak formal subschemes of $\mathfrak{S}$ defined by $t\mapsto 0$ and $t\mapsto\pi$.

For a fine $S_0$-log scheme $(Y,\mathcal{N})$ and a choice of an open covering $\{U_i\}_{i\in I}$ of $Y$ and exact closed immersions $(U_i,\mathcal{N})\hookrightarrow(\mathfrak{P}_i,\mathcal{N}_{\mathfrak{P}_i})$ into smooth weak formal $\mathfrak{S}_W$-(resp. $\mathfrak{S}_V$-)log schemes, Gro\ss e-Kl\"{o}nne defined a bounded complex of $K_0$-(resp. $K$-) vector spaces.
We call them the log rigid complexes of $(Y,\mathcal{N})$ over $K_0$ (resp. over $K$), and denote by $\RG_0((Y,\mathcal{N});\{(\mathfrak{P}_i,\mathcal{N}_{\mathfrak{P}_i})\}_{i\in I})$ (resp. $\RG_K((Y,\mathcal{N});\{(\mathfrak{P}_i,\mathcal{N}_{\mathfrak{P}_i})\}_{i\in I})$).

For a morphism $f:(Y',\mathcal{N'})\rightarrow(Y,\mathcal{N})$ of fine $S_0$-log schemes and a choice of $\{(\mathfrak{P}_i,\mathcal{N}_{\mathfrak{P}_i})\}_{i\in I}$ for $(Y,\mathcal{N})$ and $\{(\mathfrak{P}'_j,\mathcal{N}_{\mathfrak{P}'_j})\}_{j\in J}$ for $(Y',\mathcal{N}')$ as above, he also defined a complex $\RG_?(f;\{(\mathfrak{P}_i,\mathcal{N}_{\mathfrak{P}_i})\}_{i\in I},\{(\mathfrak{P}'_j,\mathcal{N}_{\mathfrak{P}'_j})\}_{j\in J})$ with maps
\begin{eqnarray*}
\RG_?((Y',\mathcal{N}');\{(\mathfrak{P}'_j,\mathcal{N}_{\mathfrak{P}'_j})\}_{j\in J})&\rightarrow&\RG_?(f;\{(\mathfrak{P}_i,\mathcal{N}_{\mathfrak{P}_i})\}_{i\in I},\{(\mathfrak{P}'_j,\mathcal{N}_{\mathfrak{P}'_j})\}_{j\in J})\\
&\leftarrow&\RG_?((Y,\mathcal{N});\{(\mathfrak{P}_i,\mathcal{N}_{\mathfrak{P}_i})\}_{i\in I})
\end{eqnarray*}
here $?$ is $0$ or $K$, and the first map is quasi-isomorphism. This gives the functoriality of log rigid complexes in the derived category.

A log scheme with boundary is a morphism $\iota:(Y,\mathcal{N}_Y)\hookrightarrow(\overline{Y},\mathcal{N}_{\overline{Y}})$ such that the underlying morphism of schemes is a schematically dense open immersion, $\mathcal{N}_{\overline{Y}}\rightarrow\iota_*\mathcal{N}_Y$ is injective, $\iota^*\mathcal{N}_{\overline{Y}}=\mathcal{N}_Y$, and $(\iota_*\mathcal{N}_Y)^{\mathrm{gp}}=\mathcal{N}_{\overline{Y}}^{\mathrm{gp}}$.
We refer \cite{Grosse3} for details on log schemes with boundary.
For an $S$-log scheme with boundary $(Y,\mathcal{N}_Y)\hookrightarrow(\overline{Y},\mathcal{N}_{\overline{N}})$ and a choice of an open covering $\{\overline{U}_i\}_{i\in I}$ of $\overline{Y}$ and boundary exact closed immersions $((U_i,\mathcal{N}_Y)\hookrightarrow(\overline{U}_i,\mathcal{N}_{\overline{Y}}))\rightarrow((\mathcal{P}_i,\mathcal{N}_{\mathcal{P}_i})\hookrightarrow(\overline{\mathcal{P}_i},\mathcal{N}_{\overline{\mathcal{P}_i}}))$, Gro\ss e-Kl\"{o}nne defined a bounded complex of $K_0$-vector spaces. We call it the log rigid complex of $(Y,\mathcal{N}_Y)\hookrightarrow(\overline{Y},\mathcal{N}_{\overline{N}})$, and denote by $\RG((Y,\mathcal{N}_Y)\hookrightarrow(\overline{Y},\mathcal{N}_{\overline{N}});\{(\mathcal{P}_i,\mathcal{N}_{\mathcal{P}_i})\hookrightarrow(\overline{\mathcal{P}_i},\mathcal{N}_{\overline{\mathcal{P}_i}})\}_{i\in I})$.

More generally, he also defined the log rigid complexes of a simplicial fine $S_0$-log scheme and of a simplicial $S$-log schemes with boundary.

\subsection{Frobenius and monodromy operator}
We use the definition of monodromy operators on log rigid cohomology due to Gro\ss e-Kl\"{o}nne \cite{Grosse1}.

\begin{definition}[strictly semistable scheme]
A $V$-scheme $\mathcal{X}$ is strictly semistable if Zariski locally it is \'{e}tale over $\Spec V[T_1,\ldots,T_n]/(T_1\cdots T_r-\pi)$ for some $0\leq r\leq n$.
\end{definition}

From now on, we use the following notation; $\mathcal{X}$ is a strictly semistable scheme over $V$ with the generic fiber $X$ and the special fiber $Y$. Let $\mathcal{N}_{\mathcal{X}}$ be the log structure on $\mathcal{X}$ defined from $Y$, and $\mathcal{N}_Y$ be its pull-back to $Y$. $(\mathfrak{X},\mathcal{N}_{\mathfrak{X}})$ is the weak completion of $(\mathcal{X},\mathcal{N}_{\mathcal{X}})$.

\begin{definition}[admissible lifting]
 An admissible lifting of $\mathcal{X}$ is a weak formal log scheme $(\mathfrak{Z},\mathcal{N}_{\mathfrak{Z}})$ over $\mathfrak{S}$ together with an isomorphism $(\mathfrak{X},\mathcal{N}_{\mathfrak{X}})\cong(\mathfrak{Z},\mathcal{N}_{\mathfrak{Z}})\times_{\mathfrak{S}}\mathfrak{S}_V$ and an endomorphism $\phi$ on $\mathfrak{Z}$ satisfying the following conditions:
\begin{enumerate}
\item The underlying weak formal scheme $\mathfrak{Z}$ is smooth over $\Spwf W$, flat over $\Spwf W[t]^{\dagger}$.
\item $Z=\mathfrak{Z}\times_{\Spwf W[t]^{\dagger}}\Spec k[t]$ is generically smooth over $\Spec k[t]$.
\item $\mathfrak{Y}=\mathfrak{Z}\times_{\Spwf W[t]^{\dagger}}\Spwf W$ is a normal crossing divisor on $\mathfrak{Z}$, and $\mathcal{N}_{\mathfrak{Z}}$ is the log structure defined from this.
\item $\phi$ is a lift of $p$-th power Frobenius on $Z$, is compatible with $\sigma$ on $\Spwf W[t]^{\dagger}$, and sends equations for the $W$-flat irreducible components to their $p$-th powers.
\end{enumerate}
\end{definition}

We follow the construction of the log rigid cohomology by Gro\ss e-Kl\"{o}nne in \cite{Grosse2}. Note that admissible liftings locally exist. Denote the irreducible components of $Y$ by $\{Y_i\}_{i\in I}$. Let $\{\mathcal{U}_h\}_{h\in H}$ be an open covering of $\mathcal{X}$ such that there exist admissible liftings $\mathfrak{Z}_h$ of $\mathcal{U}_h$ for any $h\in H$. For $h\in H$ and $i\in I$ we let $\mathfrak{Y}_{h,i}$ be the unique $W$-flat component of $\mathfrak{Y}_h$ which lifts $Y_i\cap\mathcal{U}_{h,k}$; if $Y_i\cap\mathcal{U}_{h,k}$ is empty we let $\mathfrak{Y}_{h,i}$ be the empty weak formal scheme. For $m\geq 0$ and $\alpha=(h_0,\ldots,h_m)\in H^{m+1}$, let $\mathcal{U}_{\alpha}=\bigcap_{r=0}^m\mathcal{U}_{h_r}$, and let $\mathfrak{Z}''_{\alpha}$ be the blowing up of $\prod_{r=0}^m\mathfrak{Z}_{h_r}$ along $\sum_{i\in I}\prod_{r=0}^m\mathfrak{Y}_{h_r,i}$ (products are taken over $\Spwf W$). Let $\mathfrak{Z}'_{\alpha}$ be the complement of the strict transforms of all
\[\mathfrak{Y}_{h_r,i}\times\prod_{0\leq r'\leq m,~r'\neq r}\mathfrak{Z}_{h_{r'}}\]
in $\mathfrak{Z}''_{\alpha}$, and $\mathfrak{Y}'_{\alpha}$ be its exceptional divisor. Let $\mathfrak{T}_{\alpha}$ be the blowing up of $\prod_{r=0}^m\Spwf W[t]^{\dagger}$ along $\prod_{r=0}^m\Spwf W$. The diagonal embedding $\Spwf W[t]^{\dagger}\rightarrow\prod_{r=0}^m\Spwf W[t]^{\dagger}$ lifts to an embedding $\Spwf W[t]^{\dagger}\rightarrow\mathfrak{T}_{\alpha}$. Then there exists a natural morphism $\mathfrak{Z}'_{\alpha}\rightarrow\mathfrak{T}_{\alpha}$, and $\mathfrak{Z}_{\alpha}=\mathfrak{Z}'_{\alpha}\times_{\mathfrak{T}_{\alpha}}\Spwf W[t]^{\dagger}$ is smooth over $W$, and has a relative normal crossing divisor $\mathfrak{Y}_{\alpha}=\mathfrak{Y}'_{\alpha}\times_{\mathfrak{T}_{\alpha}}\Spwf W$. Let $\mathcal{N}_{\mathfrak{Z}_{\alpha}}$ be the log structure on $\mathfrak{Z}_{\alpha}$ defined by $\mathfrak{Y}_{\alpha}$. Denote $\widetilde{\omega}^{\bullet}_{\mathfrak{Z}_{\alpha}}$ the logarithmic de Rham complex of $(\mathfrak{Z}_{\alpha},\mathcal{N}_{\mathfrak{Z}_{\alpha}})$ over $(\Spwf W,\mathrm{triv.})$. Let
\begin{eqnarray*}
\widetilde{\omega}^{\bullet}_{\mathfrak{Y}_{\alpha}}&=&\widetilde{\omega}^{\bullet}_{\mathfrak{Z}_{\alpha}}\otimes\mathcal{O}_{\mathfrak{Y}_{\alpha}}\\
\omega^{\bullet}_{\mathfrak{Y}_{\alpha}}&=&\widetilde{\omega}^{\bullet}_{\mathfrak{Y}_{\alpha}}/(\widetilde{\omega}^{\bullet-1}_{\mathfrak{Y}_{\alpha}}\wedge d\log t)\\
P_j\widetilde{\omega}^{\bullet}_{\mathfrak{Z}_{\alpha}}&=&\mathrm{Im}(\widetilde{\omega}^j_{\mathfrak{Z}_{\alpha}}\otimes\Omega^{\bullet-j}_{\mathfrak{Z}_{\alpha}}\rightarrow\widetilde{\omega}^{\bullet}_{\mathfrak{Z}_{\alpha}})\\
P_j\widetilde{\omega}^{\bullet}_{\mathfrak{Y}_{\alpha}}&=&P_j\widetilde{\omega}^{\bullet}_{\mathfrak{Z}_{\alpha}}/(\widetilde{\omega}^{\bullet}_{\mathfrak{Z}_{\alpha}}\otimes\mathcal{J}_{\mathfrak{Y}_{\alpha}})
\end{eqnarray*}
where $\Omega^{\bullet}_{\mathfrak{Z}_{\alpha}}$ is the non-logarithmic de Rham complex of $\mathfrak{Z}_{\alpha}$ over $\Spwf W$, $\mathcal{J}_{\mathfrak{Y}_{\alpha}}$ is the ideal sheaf of $\mathfrak{Y}_{\alpha}$ in $\mathfrak{Z}_{\alpha}$. 
The Steenbrink double complex $A^{\bullet,\bullet}$ on $\mathfrak{Y}_{\alpha,\cQ}$ is defined by
\[A^{i,j}_{\alpha}=\widetilde{\omega}^{i+j+1}_{\mathfrak{Y}_{\alpha,\cQ}}/P_j\widetilde{\omega}^{i+j+1}_{\mathfrak{Y}_{\alpha,\cQ}},\]
the vertical differentials $A^{i,j}_{\alpha}\rightarrow A_{\alpha}^{i+1,j}$ are induced by $(-1)^jd:\widetilde{\omega}^{i+j+1}_{\mathfrak{Y}_{\alpha}}\rightarrow\widetilde{\omega}^{i+j+2}_{\mathfrak{Y}_{\alpha}}$, and the horizontal differentials $A^{i,j}_{\alpha}\rightarrow A_{\alpha}^{i,j+1}$ are induced by $\omega\mapsto\omega\wedge d\log t$.
Let $A^{\bullet}_{\alpha}$ be the associated total complex. Let 
\begin{eqnarray*}
\mathcal{U}_m&=&\coprod_{\alpha\in H^{m+1}}\mathcal{U}_{\alpha}\\
\mathfrak{Y}_{m,\cQ}&=&\coprod_{\alpha\in H^{m+1}}\mathfrak{Y}_{\alpha,\cQ}
\end{eqnarray*}
and $A^{\bullet}_m=\coprod A^{\bullet}_{\alpha}$ be a complex of sheaves on $\mathfrak{Y}_{m,\cQ}$.

We define $\RG_{\rig}(\mathcal{X};\{\mathfrak{Z}_h\}_{h\in H})$ to be the total complex of $\Gamma(]\mathcal{U}_{\bullet,k}[_{\mathfrak{Y}_{\bullet}},\Gd_{\an}A^{\bullet}_{\bullet})$.
The diagonal actions of $\phi$ on $\mathfrak{Z}_G$ induce the Frobenius operator $\phi$ on $\RG_{\rig}(\mathcal{X};\{\mathfrak{Z}_h\}_{h\in H})$ which is a $\sigma$-semilinear bijection. The monodromy operator $N$ on $\RG_{\rig}(\mathcal{X};\{\mathfrak{Z}_h\}_{h\in H})$ is induced by $(-1)^{j+1}$ times the natural projections $A^{i,j}_{\alpha}\rightarrow A^{i-1,j+1}_{\alpha}$, and it is nilpotent. These operators satisfy the relation $N\phi=p\phi N$. In particular, $\RG_{\rig}(\mathcal{X},\{\mathfrak{Z}_h\}_{h\in H})$ is an object in $C^b_{\rig,K_0}$.

The natural map $\widetilde{\omega}^{\bullet}_{\mathfrak{Y}_{\alpha,\cQ}}\rightarrow A_{\alpha}^{\bullet,0}:\eta\mapsto\eta\wedge d\log t$ induces a quasi-isomorphisms $\omega^{\bullet}_{\mathfrak{Y}_{\alpha,\cQ}}\rightarrow A_{\alpha}^{\bullet}$ and
\begin{equation}\label{rig0}
\RG_0((Y,\mathcal{N}_Y),\{(\mathfrak{Y}_h,\mathcal{N}_{\mathfrak{Y}_h})\}_{h\in H})\rightarrow\RG_{\rig}(\mathcal{X},\{\mathfrak{Z}_h\}_{h\in H})
\end{equation}

\subsection{Functoriality; construction of maps}
Let $\mathcal{X}$ and $\mathcal{X}'$ be strictly semistable schemes over $V$, $f:\mathcal{X}'\rightarrow\mathcal{X}$ be a $V$-morphism. Denote the irreducible components of $Y$ and $Y'$ by $\{Y_i\}_{i\in I}$ and $\{Y'_j\}_{j\in J}$ respectively. Let
\[J_i=\{j\in J\mid f(Y'_j)\subset Y_i\},~I_0=\{i\in I\mid J_i\neq\emptyset\}.\]
Then we have $J=\coprod_{i\in I_0}J_i$.
Take an open covering $\{\mathcal{U}_h\}_{h\in H}$ of $\mathcal{X}$ with admissible liftings $\{\mathfrak{Z}_h\}_{h\in H}$ and an open covering $\{\mathcal{U}'_g\}_{g\in G}$ of $\mathcal{X}'$ with admissible liftings $\{\mathfrak{Z}'_g\}_{g\in G}$. Let $L=G\times H$.
For $\ell=(g,h)\in L$, we write $\mathcal{U}^{\times}_{\ell}=\mathcal{U}'_g\cap f^{-1}(\mathcal{U}_h),~\mathfrak{Z}_{\ell}=\mathfrak{Z}_h$, and $\mathfrak{Z}'_{\ell}=\mathfrak{Z}'_{g}$.
For $i\in I$, $j\in J$, and $\ell=(g,h)\in L$, let $\mathfrak{Y}_{\ell,i}$ and $\mathfrak{Y}'_{\ell,j}$ be the unique $W$-flat components of $\mathfrak{Y}_{\ell}$ and $\mathfrak{Y}'_{\ell}$ which lift $Y_i\cap \mathcal{U}_{h,k}$ and $Y'_j\cap \mathcal{U}_{g,k}$ respectively; if $Y_i\cap \mathcal{U}_{h,k}$ is empty we let $\mathfrak{Y}_{\ell,i}$ be the empty weak formal scheme.

For $m\geq 0$ and $\alpha=(\ell_0,\ldots,\ell_m)\in L^{m+1}$, let $\mathcal{J}_{\alpha}$ be the ideal of
\[\sum_{i\in I_0,~j\in J_i}\prod_{r=0}^m(\mathfrak{Y}_{\ell_r,i}\times\mathfrak{Y}'_{\ell_r,j})\subset\prod_{r=0}^m(\mathfrak{Z}_{\ell_r}\times\mathfrak{Z}'_{\ell_r}),\]
and $\mathfrak{P}''_{\alpha}$ be the blowing-up along $\mathcal{J}_{\alpha}$.
Let $\mathcal{I}_{\alpha}$ be the ideal of
\[\prod_{r=0}^m(\Spwf W\times\Spwf W)\subset\prod_{r=0}^m(\Spwf W[t_{\ell_r}]^{\dagger}\times\Spwf W[s_{\ell_r}]^{\dagger}),\]
and $\mathfrak{T}_{\alpha}$ be the blowing-up along $\mathcal{I}_{\alpha}$, here $t_{\ell_r}$'s and $s_{\ell_r}$'s are indeterminates.
Then the diagonal embedding $\Spwf W[t]^{\dagger}\rightarrow\prod_{r=0}^m(\Spwf W[t_{\ell_r}]^{\dagger}\times\Spwf W[s_{\ell_r}]^{\dagger})$ lifts to $\Spwf W[t]^{\dagger}\rightarrow\mathfrak{T}_{\alpha}$.
Denote
\[q_{\alpha}:\prod_{r=0}^m(\mathfrak{Z}_{\ell_r}\times\mathfrak{Z}'_{\ell_r})\rightarrow\prod_{r=0}^m(\Spwf W[t_{\ell_r}]^{\dagger}\times\Spwf W[s_{\ell_r}]^{\dagger})\]
the natural morphism, and let
\[\theta_{\alpha}:\bigoplus_{n\geq 0}q_{\alpha}^{-1}(\mathcal{I}_{\alpha})^n\rightarrow\bigoplus_{n\geq 0}\mathcal{J}_{\alpha}^n\]
be the induced homomorphism of graded rings. Let $\mathbb{V}(\theta_{\alpha}(t_{\ell_r}),\theta_{\alpha}(s_{\ell_r}))_{r=0}^m$ be the closed weak formal subscheme of $\mathfrak{P}''_{\alpha}$ defined by all images of degree-$1$ elements $t_{\ell_r}$ and $s_{\ell_r}$ by $\theta_{\alpha}$.
Let $\mathfrak{P}'_{\alpha}$ be the complement of $\mathbb{V}(\theta_{\alpha}(t_{\ell_r}),\theta_{\alpha}(s_{\ell_r}))_{r=0}^m$ and the strict transforms of
\begin{eqnarray*}
(\mathfrak{Y}_{\ell_r,i}\times\mathfrak{Z}'_{\ell_r})\times\prod_{0\leq r'\leq m,~r'\neq r}(\mathfrak{Z}_{\ell_{r'}}\times\mathfrak{Z}'_{\ell_{r'}})\\
(\mathfrak{Z}_{\ell_r}\times\mathfrak{Y}'_{\ell_r,j})\times\prod_{0\leq r'\leq m,~r'\neq r}(\mathfrak{Z}_{\ell_{r'}}\times\mathfrak{Z}'_{\ell_{r'}})
\end{eqnarray*}
for all $i\in I$ and $j\in J$ in $\mathfrak{P}''_{\alpha}$. 
Then we have a natural morphism $\mathfrak{P}'_{\alpha}\rightarrow\mathfrak{T}_{\alpha}$.
Let $\mathfrak{P}_{\alpha}=\mathfrak{P}'_{\alpha}\times_{\mathfrak{T}_{\alpha}}\Spwf W[t]^{\dagger}$. Then the exceptional divisor $\mathfrak{Q}_{\alpha}$ is a normal crossing divisor on $\mathfrak{P}_{\alpha}$.
As the construction of $\mathfrak{Y}_{\bullet,\cQ}$ and $A_{\bullet}^{\bullet}$ in last subsection, we can define a simplicial dagger space $\mathfrak{Q}_{\bullet,\cQ}$ and complexes $B_{\bullet}^{\bullet}$ of sheaves. We define $\RG_{\rig}(f;\{\mathfrak{Z}_h\}_{h\in H},\{\mathfrak{Z}'_g\}_{g\in G})$ to be the total complex of $\Gamma(]\mathcal{U}^{\times}_{\bullet,k}[_{\mathfrak{Q}_{\bullet}},\Gd_{\an}B^{\bullet}_{\bullet})$, then this is an object in $C^b_{\rig,K_0}$.
Moreover we have natural maps 
\begin{equation}\label{functorial}
\RG_{\rig}(\mathcal{X},\{\mathfrak{Z}_h\}_{h\in H})\rightarrow\RG_{\rig}(f;\{\mathfrak{Z}_{h}\}_{h\in H},\{\mathfrak{Z}'_g\}_{g\in G})\xleftarrow[]{(*)}\RG_{\rig}(\mathcal{X}';\{\mathfrak{Z}'_g\}_{g\in G})
\end{equation}
in $C^b_{\rig,K_0}$.

\begin{lemma}\label{funclemma}
The map $(*)$ in \eqref{functorial} is a quasi-isomorphism.
\end{lemma}

\begin{proof}
Since the following diagram \eqref{relpolyproof} is commutative and vertical arrows are quasi-isomorphic, it suffices to show that $]\mathcal{U}^{\times}_{\bullet,k}[_{\mathfrak{Q}_{\bullet}}\rightarrow]\mathcal{U}'_{\bullet,k}[_{\mathfrak{Y}'_{\bullet}}$ is a relative open polydisk.
\begin{equation}\label{relpolyproof}\xymatrix{
\Gamma(]\mathcal{U}'_{\bullet,k}[_{\mathfrak{Y}'_{\bullet}},\Gd_{\an}\omega^{\bullet}_{\mathfrak{Y}'_{\bullet,\cQ}})\ar[d]\ar[r]&\Gamma(]\mathcal{U}^{\times}_{\bullet,k}[_{\mathfrak{Q}_{\bullet}},\Gd_{\an}\omega^{\bullet}_{\mathfrak{Q}_{\bullet,\cQ}})\ar[d]\\
\RG_{\rig}(\mathcal{X}',\{\mathfrak{Z}'_{g}\}_{g\in G})\ar[r]^<<<<<<<{(*)}&\RG_{\rig}(f;\{\mathfrak{Z}_h\}_{h\in H},\{\mathfrak{Z}'_g\}_{g\in G})
}\end{equation}
Since the statement is local, we may assume that there are finite sets $\bar{I}$ and $\bar{J}$ which contain $I$ and $J$ respectively, and \'{e}tale morphisms
\[\mathfrak{Z}_{\ell}\rightarrow\Spwf W[T_{\ell,i}]^{\dagger}_{i\in\bar{I}},\hspace{10pt}\mathfrak{Z}'_{\ell}\rightarrow\Spwf W[S_{\ell,j}]^{\dagger}_{j\in\bar{J}}\]
such that $\mathfrak{Y}_{\ell,i}$ and $\mathfrak{Y}'_{\ell,j}$ are defined by $T_{\ell,i}$ and $S_{\ell,j}$ respectively for all $\ell\in L$, $i\in I$, and $j\in J$. Then, for fixed any $0\leq q\leq m$, $i_0\in I_0$, and $j_0\in J$, $\mathfrak{Q}_{\alpha}$ and $\mathfrak{Y}'_{\alpha}$ are \'{e}tale over
\[\mathfrak{R}_{\alpha}=\Spwf\frac{W[T_{\ell_r,a}^{\pm 1},T_{\ell_r,a'},S_{\ell_q,b},S_{\ell_r,b'},(\frac{S_{\ell_{r'},c}}{S_{\ell_q,c}})^{\pm 1},(\frac{\prod_{j\in J_d}S_{\ell_q,j}}{T_{\ell_q,d}})^{\pm 1}]^{\dagger}{\tiny\begin{array}{l}0\leq r\leq m,~0\leq r'\leq m,~r\neq r'\\ a\in I\setminus I_0,~a'\in\bar{I}\setminus I\\ b\in J,~b'\in\bar{J}\setminus J\\ c\in J\setminus\{j_0\},~d\in I_0\setminus\{i_0\}\end{array}}}{(\prod_{j\in J}S_{\ell_q,j})}\]
and
\[\mathfrak{V}'_{\alpha}=\Spwf\frac{W[S_{\ell_q,b},S_{\ell_r,b'},(\frac{S_{\ell_{r'},c}}{S_{\ell_q,c}})^{\pm 1}]^{\dagger}{\tiny\begin{array}{l}0\leq r\leq m,~0\leq r'\leq m,~r\neq r'\\ b\in J,~b'\in\bar{J}\setminus J\\ c\in J\setminus\{j_0\}\end{array}}}{(\prod_{j\in J}S_{\ell_q,j})}\]
respectively.
Let
\[R=\Spec k[S_{\ell_q,j}]_{j\in J}/(\prod_{j\in J}S_{\ell_q,j}).\]
Then we have a commutative diagram
\[\xymatrix{
]\mathcal{U}^{\times}_{\bullet,k}[_{\mathfrak{Q}_{\bullet}}\ar[r]\ar[d]&]\mathcal{U}'_{\bullet,k}[_{\mathfrak{Y}'_{\bullet}}\ar[d]\\
]R[_{\mathfrak{R}_{\bullet}}\ar[r]&]R[_{\mathfrak{V}'_{\bullet}}
}\]
in which the vertical arrows are isomorphisms.
Since $\mathfrak{R}_{\alpha}\rightarrow\mathfrak{V}'_{\alpha}$ is smooth, $]R[_{\mathfrak{R}_{\alpha}}\rightarrow]R[_{\mathfrak{V}'_{\alpha}}$ is a relative open polydisk.
\end{proof}

\subsection{Functoriality; composition}
Next we consider about the composition.
Let $\mathcal{X}''\xrightarrow[]{g}\mathcal{X}'\xrightarrow[]{f}\mathcal{X}$ be the morphisms of strictly semistable schemes over $V$. 
Let $\{Y_i\}_{i\in I}$, $\{Y'_j\}_{j\in J}$, and $\{Y''_q\}_{q\in Q}$ be the irreducible components of $Y$, $Y'$, and $Y''$. Let
\begin{eqnarray*}J_i&=&\{j\in J\mid f(Y'_j)\subset Y_i\}\\
Q_i&=&\{q\in Q\mid f\circ g(Y''_q)\subset Y_i\}\\
Q_j&=&\{q\in Q\mid g(Y''_q)\subset Y'_j\}
\end{eqnarray*}
for $i\in I$ and $j\in J$, and
\begin{eqnarray*}
I_0&=&\{i\in I\mid J_i\neq 0\}\\
I_{00}&=&\{i\in I\mid Q_i\neq 0\}\\
J_0&=&\{j\in J\mid Q_j\neq 0\}\\
J_{i,0}&=&J_i\cap J_0.
\end{eqnarray*}

Take open coverings $\{\mathcal{U}_h\}_{h\in H}$, $\{\mathcal{U}'_g\}_{g\in G}$, and $\{\mathcal{U}''_e\}_{e\in E}$ of $\mathcal{X}$, $\mathcal{X}'$, and $\mathcal{X}''$ with admissible coverings $\{\mathfrak{Z}_h\}_{h\in H}$, $\{\mathfrak{Z}'_g\}_{g\in G}$, $\{\mathfrak{Z}''_e\}_{e\in E}$. Let
\[\mathcal{L}=E\times G\times H.\]
For $\ell=(e,g,h)\in\mathcal{L}$, we write $\mathfrak{Z}_{\ell}=\mathfrak{Z}_h$, $\mathfrak{Z}'_{\ell}=\mathfrak{Z}'_g$, and $\mathfrak{Z}''_{\ell}=\mathfrak{Z}''_e$.

For $m\geq 0$ and $\alpha=(\ell_0,\ldots,\ell_m)\in\mathcal{L}^{m+1}$, let $\widetilde{\mathcal{J}}_{\alpha}$ be the ideal of
\[\sum_{i\in I_{00},~j\in J_{i,0},~q\in Q_j}\prod_{r=0}^m(\mathfrak{Y}_{\ell_r,i}\times\mathfrak{Y}'_{\ell_r,j}\times\mathfrak{Y}''_{\ell_r,q})\subset\prod_{r=0}^m(\mathfrak{Z}_{\ell_r}\times\mathfrak{Z}'_{\ell_r}\times\mathfrak{Z}''_{\ell_r})\]
and $\widetilde{\mathfrak{P}}''_{\alpha}$ be the blowing-up along $\widetilde{\mathcal{J}}_{\alpha}$.
Let $\widetilde{\mathcal{I}}_{\alpha}$ be the ideal of
\[\prod_{r=0}^m(\Spwf W\times\Spwf W\times\Spwf W)\subset\prod_{r=0}^m(\Spwf W[t_{\ell_r}]^{\dagger}\times\Spwf W[s_{\ell_r}]^{\dagger}\times\Spwf W[u_{\ell_r}]^{\dagger}),\]
and $\widetilde{\mathfrak{T}}_{\alpha}$ be the blowing-up along $\widetilde{\mathcal{I}}_{\alpha}$, here $t_{\ell_r}$'s, $s_{\ell_r}$'s, and $u_{\ell_r}$'s are indeterminates. Then the diagonal embedding $\Spwf W[t]^{\dagger}\rightarrow\prod_{r=0}^m(\Spwf W[t_{\ell_r}]^{\dagger}\times\Spwf W[s_{\ell_r}]^{\dagger}\times\Spwf W[u_{\ell_r}]^{\dagger})$ lifts to $\Spwf W[t]^{\dagger}\rightarrow\widetilde{\mathfrak{T}}_{\alpha}$. Denote
\[\widetilde{q}_{\alpha}:\prod_{r=0}^m(\mathfrak{Z}_{\ell_r}\times\mathfrak{Z}'_{\ell_r}\times\mathfrak{Z}''_{\ell_r})\rightarrow\prod_{r=0}^m(\Spwf W[t_{\ell_r}]^{\dagger}\times\Spwf W[s_{\ell_r}]^{\dagger}\times\Spwf W[u_{\ell_r}]^{\dagger})\]
be the natural morphism, and let
\[\widetilde{\theta}_{\alpha}:\bigoplus_{n\geq 0}\widetilde{q}^{-1}_{\alpha}(\widetilde{\mathcal{I}}_{\alpha})^n\rightarrow\bigoplus_{n\geq 0}\widetilde{\mathcal{J}}_{\alpha}^n\]
be the induced homomorphism of graded rings. Let $\mathbb{V}(\widetilde{\theta}_{\alpha}(t_{\ell_r}),\widetilde{\theta}_{\alpha}(s_{\ell_r}),\widetilde{\theta}_{\alpha}(u_{\ell_r}))_{r=0}^m$ be the closed weak formal subscheme of $\widetilde{\mathfrak{P}}''_{\alpha}$ defined by all images of degree-$1$ elements $t_{\ell_r},s_{\ell_r}$, and $u_{\ell_r}$ by $\widetilde{\theta}_{\alpha}$.
Let $\widetilde{\mathfrak{P}}'_{\alpha}$ be the complement of $\mathbb{V}(\widetilde{\theta}_{\alpha}(t_{\ell_r}),\widetilde{\theta}_{\alpha}(s_{\ell_r}),\widetilde{\theta}_{\alpha}(u_{\ell_r}))_{r=0}^m$ and the strict transforms of
\begin{eqnarray*}
(\mathfrak{Y}_{\ell_r,i}\times\mathfrak{Z}'_{\ell_r}\times\mathfrak{Z}''_{\ell_r})\times\prod_{0\leq r'\leq m,~r'\neq r}(\mathfrak{Z}_{\ell_{r'}}\times\mathfrak{Z}'_{\ell_{r'}}\times\mathfrak{Z}''_{\ell_{r'}})\\
(\mathfrak{Z}_{\ell_r}\times\mathfrak{Y}'_{\ell_r,j}\times\mathfrak{Z}''_{\ell_r})\times\prod_{0\leq r'\leq m,~r'\neq r}(\mathfrak{Z}_{\ell_{r'}}\times\mathfrak{Z}'_{\ell_{r'}}\times\mathfrak{Z}''_{\ell_{r'}})\\
(\mathfrak{Z}_{\ell_r}\times\mathfrak{Z}'_{\ell_r}\times\mathfrak{Y}''_{\ell_r,q})\times\prod_{0\leq r'\leq m,~r'\neq r}(\mathfrak{Z}_{\ell_{r'}}\times\mathfrak{Z}'_{\ell_{r'}}\times\mathfrak{Z}''_{\ell_{r'}})
\end{eqnarray*}
for all $i\in I$, $j\in J$, and $q\in Q$ in $\widetilde{\mathfrak{P}}''_{\alpha}$. Then we have a natural morphism $\widetilde{\mathfrak{P}}'_{\alpha}\rightarrow\widetilde{\mathfrak{T}}_{\alpha}$.
Let $\widetilde{\mathfrak{P}}_{\alpha}=\widetilde{\mathfrak{P}}_{\alpha}\times_{\widetilde{\mathfrak{T}}_{\alpha}}\Spwf W[t]^{\dagger}$. Then the exceptional divisor $\widetilde{\mathfrak{Q}}_{\alpha}$ is a normal crossing divisor on $\widetilde{\mathfrak{P}}_{\alpha}$. We can define a simplicial dagger space $\widetilde{\mathfrak{Q}}_{\bullet,\cQ}$, complexes of sheaves $\widetilde{B}_{\bullet}^{\bullet}$, and an element $\RG_{\rig}(f,g;\{\mathfrak{Z}_h\}_{h\in H},\{\mathfrak{Z}'_g\}_{g\in G},\{\mathfrak{Z}''_e\}_{e\in E})$ in $C^b_{\rig,K_0}$ as before. Moreover we have a commutative diagram
\begin{equation}\label{bigdiag}
{\small\xymatrix{
&\RG_{\rig}(\mathcal{X};\{\mathfrak{Z}_h\}_h)\ar[rd]\ar[ldd]\ar[dd]&\\
&&\RG_{\rig}(f;\{\mathfrak{Z}_h\}_h,\{\mathfrak{Z}'_g\}_g)\ar[ld]\\
\RG_{\rig}(f\circ g;\{\mathfrak{Z}_h\}_h,\{\mathfrak{Z}''_e\}_e)\ar[r]^<<<<<{\cong}_<<<<{(\sharp)}&\RG_{\rig}(f,g;\{\mathfrak{Z}_h\}_h,\{\mathfrak{Z}'_g\}_g,\{\mathfrak{Z}''_e\}_e)&\RG_{\rig}(\mathcal{X}';\{\mathfrak{Z}'_g\}_g)\ar[u]_{\cong}\ar[d]\ar[l]\\
&&\RG_{\rig}(g;\{\mathfrak{Z}'_g\}_g,\{\mathfrak{Z}''_e\}_e)\ar[lu]^{(\sharp)}_{\cong}\\
&\RG_{\rig}(\mathcal{X}'';\{\mathfrak{Z}''_e\}_e)\ar[luu]^{\cong}\ar[ru]_{\cong}\ar[uu]^{\cong}_{(\sharp)}&
}}\end{equation}
in $C^b_{\rig,K_0}$.
So the construction in the last subsection is compatible with composition.

\begin{remark}
That the maps with $(\sharp)$ in the diagram \eqref{bigdiag} are quasi-isomorphic can be showed by local description as the proof of Lemma \ref{funclemma}. 
\end{remark}

Considering the case $f$ is identity and the case $g$ is identity, we can show that $\RG_{\rig}(\mathcal{X},\{\mathfrak{Z}_h\}_{h\in H})$ and the construction \eqref{functorial} are independent of the choice of admissible liftings up to canonical quasi-isomorphisms. In other words, $\RG_{\rig}$ gives a functor from the category of strictly semistable schemes over $V$ to the derived category of $C^b_{\rig,K_0}$.

\section{Log rigid syntomic cohomology}
\subsection{Definition}\label{section4.1}

Let $\mathcal{X}$ be a strictly semistable scheme over $V$. Gro\ss e-Kl\"{o}nne proved that the base change of the log rigid complex over $K_0$ is quasi-isomorphic to that over $K$ (\cite{Grosse1} Theorem 3.4). This quasi-isomorphism depends on the choice of an uniformizer of $K$. We follow his construction. We refer the second section of \cite{Grosse1} for more details.

Let $\mathcal{N}_Y$ be the log structure on $Y$ defined as the inverse image of $\mathcal{N}_{\mathcal{X}}$.
Let $\{Y_i\}_{i\in I}$ be the irreducible components of $Y$. Choose an open covering $\{\mathcal{U}_h\}_{h\in H}$ of $\mathcal{X}$ with admissible liftings $\{\mathfrak{Z}_h\}_{h\in H}$. For $i\in I$ let $\mathcal{N}_j$ be the preimage of ideal of $Y_i$ in $Y$ by $\mathcal{N}_Y\rightarrow\mathcal{O}_Y$. Let $\mathcal{L}_j$ be the line bundle on $Y$ associated to a principal homogeneous space $\mathcal{N}_j$ over $\mathcal{O}_Y^{\times}$.

For a non-empty subset $J\subset I$, let $M_J=\bigcap_{j\in J}Y_j$, and let $\mathcal{N}_{M_J}$ be the inverse image of $\mathcal{N}_Y$ on $M_J$.
For $j\in J$ let
\begin{eqnarray*}
\mathcal{L}_J^j&=&\mathcal{L}_j\otimes_{\mathcal{O}_Y}\mathcal{O}_{M_J}\\
V_J^j&=&\Spec(\mathrm{Sym}_{\mathcal{O}_{M_J}}\mathcal{L}_J^j)\\
P_J^j&=&\Proj(\mathrm{Sym}_{\mathcal{O}_{M_J}}(\mathcal{O}_{M_J}\oplus\mathcal{L}_J^j)).
\end{eqnarray*}

For a non-empty subset $J'\subset J$ let
\[V_J^{J'}=\prod_{j\in J'}^{M_J}V_J^j,\hspace{10pt}P_J^{J'}=\prod_{j\in J'}^{M_J}P_J^j\]
and write $V_J=V_J^J$, $P_J=P_J^J$. Then there is a natural open immersion $V_J^{J'}\hookrightarrow P_J^{J'}$.
Let $N^j_{J,\infty}$ be the divisor on $P_J$ which is the pull-back of the divisor $P_J^j\setminus V_J^j$ on $P_J^j$. Let  $N^j_{J,0}$ be the divisor on $P_J$ which is the pull-back of the zero section divisor $M_J\rightarrow P_J^j$ on $P_J^j$. Let $D_J$ be the divisor on $P_J$ which is the pull-back of the divisor $M_J\cap\bigcup_{j\in I\setminus J}Y_j$ on $M_J$. Let $\mathcal{N}_{P_J}$ be the log structure on $P_J$ defined by the normal crossing divisor $\bigcup_{j\in J}N_{J,\infty}^j\cup\bigcup_{j\in J}N_{J,0}^j\cup D_J$. Let $\mathcal{N}_{V_J^{J'}}$ and $\mathcal{N}_{P_J^{J'}}$ be the inverse images of $\mathcal{N}_{P_J}$ on $V_J^{J'}$ and $P_J^{J'}$.
For $m\geq 0$ let
\[\Lambda_m=\{\lambda=(J_0(\lambda),\ldots,J_m(\lambda))\mid\emptyset\neq J_0(\lambda)\subsetneq\cdots\subsetneq J_m(\lambda)\subset I\}.\]
For $\lambda\in\Lambda_m$ let
\[V_{\lambda}=V_{J_m(\lambda)}^{J_0(\lambda)},\hspace{10pt}P_{\lambda}=P_{J_m(\lambda)}^{J_0(\lambda)}.\]
Let
\begin{eqnarray*}
(M_m,\mathcal{N}_{M_m})&=&\coprod_{\lambda\in\Lambda_m}(M_{J_m(\lambda)},\mathcal{N}_{M_{J_m(\lambda)}})\\
(V_m,\mathcal{N}_{V_m})&=&\coprod_{\lambda\in\Lambda_m}(V_{\lambda},\mathcal{N}_{V_{\lambda}})\\
(P_m,\mathcal{N}_{P_m})&=&\coprod_{\lambda\in\Lambda_m}(P_{\lambda},\mathcal{N}_{P_{\lambda}}).
\end{eqnarray*}

Then we have a simplicial fine $S_0$-log scheme $(M_{\bullet},\mathcal{N}_{M_{\bullet}})$ and a simplicial $S$-log scheme with boundary $(V_{\bullet},\mathcal{N}_{V_{\bullet}})\hookrightarrow(P_{\bullet},\mathcal{N}_{P_{\bullet}})$.

Take an open covering $\{\mathcal{U}_h\}_{h\in H}$ of $\mathcal{X}$ and admissible liftings $\mathfrak{Z}_h$'s of $\mathcal{U}_h$'s, and let $\mathfrak{U}_h$ be the weak completion of $\mathcal{U}_h$. Let $\mathfrak{Y}_h=\mathfrak{Z}_h\times_{\Spwf W[t]^{\dagger}}\Spwf W$. For $m\geq 0$ and $h\in H$, let $\mathfrak{Y}^{(m)}_h=\coprod_{\lambda\in\Lambda_m}\mathfrak{Y}_h$, $\mathfrak{U}_{m,h}=\coprod_{\lambda\in\Lambda_m}\mathfrak{U}_h$, and $P_{m,h}=P_m\times_Y\mathcal{U}_{h,k}$. Refining the covering $\{\mathcal{U}_h\}_{h\in H}$ if necessary, we can take boundary exact closed immersions $((V_{m,h},\mathcal{N}_{V_{m,h}})\hookrightarrow(P_{m,h},\mathcal{N}_{P_{m,h}}))\rightarrow((\mathcal{V}_{m,h},\mathcal{N}_{\mathcal{V}_{m,h}})\hookrightarrow(\mathcal{P}_{m,h},\mathcal{N}_{\mathcal{P}_{m,h}}))$ into smooth $\widetilde{S}$-log schemes with boundary.
Then we have natural quasi-isomorphisms

\begin{eqnarray}\label{0boundary}
\nonumber&&\RG_0((Y,\mathcal{N}_Y);\{(\mathfrak{Y}_h,\mathcal{N}_{\mathfrak{Y}_h})\}_{h\in H})\\
\nonumber&\rightarrow&\RG_0((M_{\bullet},\mathcal{N}_{M_{\bullet}});\{(\mathfrak{Y}^{(m)}_h,\mathcal{N}_{\mathfrak{Y}^{(m)}_h})\}_{h\in H,~m\geq 0})\\
&\rightarrow&\RG_0(\mathrm{id}_{M_{\bullet}};\{(\mathfrak{Y}_h^{(m)},\mathcal{N}_{\mathfrak{Y}_h^{(m)}})\}_{h\in H,~m\geq 0},\{(\mathcal{V}_{m,h,W},\mathcal{N}_{\mathcal{V}_{m,h,W}})\}_{h\in H,~m\geq 0})\\
\nonumber&\leftarrow&\RG_0((M_{\bullet},\mathcal{N}_{M_{\bullet}});\{(\mathcal{V}_{m,h,W},\mathcal{N}_{\mathcal{V}_{m,h,W}})\}_{h\in H,~m\geq 0})\\
\nonumber&\leftarrow&\RG((V_{\bullet},\mathcal{N}_{V_{\bullet}})\hookrightarrow(P_{\bullet},\mathcal{N}_{P_{\bullet}});\{(\mathcal{V}_{m,h},\mathcal{N}_{\mathcal{V}_{m,h}})\hookrightarrow(\mathcal{P}_{m,h},\mathcal{N}_{\mathcal{P}_{m,h}})\}_{h\in H,~m\geq 0})
\end{eqnarray}
and
\begin{eqnarray}\label{boundaryK}
\nonumber&&\RG((V_{\bullet},\mathcal{N}_{V_{\bullet}})\hookrightarrow(P_{\bullet},\mathcal{N}_{P_{\bullet}});\{(\mathcal{V}_{m,h},\mathcal{N}_{\mathcal{V}_{m,h}})\hookrightarrow(\mathcal{P}_{m,h},\mathcal{N}_{\mathcal{P}_{m,h}})\}_{h\in H,~m\geq 0})\otimes K\\
\nonumber&\rightarrow&\RG_K((M_{\bullet},\mathcal{N}_{M_{\bullet}});\{(\mathcal{V}_{m,h,V},\mathcal{N}_{\mathcal{V}_{m,h,V}})\}_{h\in H,~m\geq 0})\\
&\rightarrow&\RG_K(\mathrm{id}_{M_{\bullet}};\{(\mathfrak{U}_{m,h},\mathcal{N}_{\mathfrak{U}_{m,h}})\}_{h\in H,~m\geq 0},\{(\mathcal{V}_{m,h,V},\mathcal{N}_{\mathcal{V}_{m,h,V}})\}_{h\in H,~m\geq 0})\\
\nonumber&\leftarrow&\RG_K((M_{\bullet},\mathcal{N}_{M_{\bullet}});\{(\mathfrak{U}_{m,h},\mathcal{N}_{\mathfrak{U}_{m,h}})\}_{h\in H,~m\geq 0})\\
\nonumber&\leftarrow&\RG_K((Y,\mathcal{N}_Y);\{(\mathfrak{U}_h,\mathcal{N}_{\mathfrak{U}_h})\}_{h\in H})\\
\nonumber&\leftarrow&\RG_K((Y,\mathcal{N}_Y);\{(\mathfrak{X},\mathcal{N}_{\mathfrak{X}})\}).
\end{eqnarray}

And we have natural maps
\begin{eqnarray}\label{KdR}
\nonumber&&\RG_K((Y,\mathcal{N}_Y);\{(\mathfrak{X},\mathcal{N}_{\mathfrak{X}})\})\xleftarrow[]{(\diamondsuit)}\Gamma(\mathfrak{X}_K,\Gd_{\an}\Omega^{\log,\bullet}_{\mathfrak{X}})\leftarrow\Gamma(\mathfrak{X}_K,\Gd_{\an}\Omega^{\bullet}_{\mathfrak{X}})\\
&=&\Gamma(X^{\an},i_*\Gd_{\an}\Omega^{\bullet}_{\mathfrak{X}})\leftarrow\Gamma(X^{\an},\Gd_{\an}\Omega^{\bullet}_{X^{\an}})\leftarrow\Gamma(X^{\an},\Gd_{\an}w^*\Omega^{\bullet}_X)\\
\nonumber&\xleftarrow[]{(\flat)}&\Gamma(X,\Gd_{\an+\mathrm{Zar}}\Omega^{\bullet}_X)\xrightarrow[]{(\natural)}\Gamma(X,\Gd_{\mathrm{Zar}}\Omega_X^{\bullet})\rightarrow\RG_{\dR}(X/K).
\end{eqnarray}
Here, $\Gd_{\mathrm{Zar}}$ and $\Gd_{\an+\mathrm{Zar}}$ are Godement resolutions associated to $P_{\mathrm{Zar}}(X)\rightarrow X_{\mathrm{Zar}}$ and $P_{\mathrm{Zar}}(X)\coprod Pt(X^{\an})\rightarrow X_{\mathrm{Zar}}$, where $P_{\mathrm{Zar}}(X)$ is the set of Zariski points of $X$ with discrete topology, $X_{\mathrm{Zar}}$ is the Zariski site of $X$. The maps $(\flat)$ and $(\natural)$ are obtained by \cite{CCM} Proposition 4.9. $\RG_{\dR}(X/K)$ is the derived de Rham cohomology of $X$, which is an object in $C^b_{\dR,K}$ (\cite{Beilinson2} section 3.4). Note that these maps are always quasi-isomorphic except $(\diamondsuit)$.

By quasi push-out construction (\cite{CCM} Remark 2.12), we obtain the $p$-adic Hodge complex $\RG_{\Hdg}(\mathcal{X})$ associated to $\mathcal{X}$ from the maps \eqref{rig0}, \eqref{0boundary}, \eqref{boundaryK}, and \eqref{KdR}. By the functoriality of $\RG_{\rig}$ proved in previous subsections and the functoriality of $\RG_0$ and $\RG_K$, it is independent of all of the choice of data up to canonical quasi-isomorphisms, and defines a functor from the category of strictly semistable schemes over $V$ to $\pHD$.

We consider the condition
\begin{center}
(HK): the map $(\diamondsuit)$ in \eqref{KdR} is quasi-isomorphic.
\end{center}
When $\mathcal{X}$ satisfies (HK), $\RG_{\Hdg}(\mathcal{X})$ represents an object in $\widetilde{\pHD}$.  We define the $p$-adic Hodge cohomology group of $\mathcal{X}$ by
\[H^i_{\Hdg}(\mathcal{X},n)=H^i(\RG_{\Hdg}(\mathcal{X})(n))\]
which is an object in $\MF$.

\begin{definition}[log rigid syntomic cohomology]
For a strictly semistable scheme $\mathcal{X}$ over $V$, we define the log rigid syntomic cohomology group by
\[H^i_{\syn}(\mathcal{X},n)=\Ext^i_{\pHD}(K_0,\RG_{\Hdg}(\mathcal{X})(n))=\Hom_{\pHD}(K_0,\RG_{\Hdg}(\mathcal{X})(n)[i]).\]
\end{definition}

\subsection{Properties}
For a strictly semistable scheme $\mathcal{X}$ over $V$, put
\begin{eqnarray*}
H^i_{\mathcal{A}}(\mathcal{X},n)&=&H^i(\mathcal{A}^{\bullet}_0(\RG_{\Hdg}(\mathcal{X})(n)))\\
&=&H^i_{\mathrm{rig}}(Y/K_0)\oplus F^nH^i_{\mathrm{dR}}(X/K)\\
H^i_{\mathcal{B}}(\mathcal{X},n)&=&H^i(\mathcal{B}^{\bullet}_0(\RG_{\Hdg}(\mathcal{X})(n)))\\
&=&H^i_{\mathrm{rig}}(Y/K_0)\oplus H^i_{\mathrm{rig}}(Y/K_0)\oplus H^i_{\mathrm{rig}}(Y/K)\\
H^i_{\mathcal{C}}(\mathcal{X},n)&=&H^i(\mathcal{C}^{\bullet}_0(\RG_{\Hdg}(\mathcal{X})(n)))\\
&=&H^i_{\mathrm{rig}}(Y/K_0)\\
H^i_{\alpha}(\mathcal{X},n)&=&H^i(\Cone\Phi_0(\RG_{\Hdg}(\mathcal{X},n)))\\
H^i_{\beta}(\mathcal{X},n)&=&H^i(\Cone\Psi_0(\RG_{\Hdg}(\mathcal{X},n))).
\end{eqnarray*}

The following proposition is concluded easily from the definition.
\begin{proposition}
For strictly semistable scheme $\mathcal{X}$ over $V$, there exist long exact sequences as follows.
\[\xymatrix{
&&\vdots\ar[d]&\vdots\ar[d]&&\\
\cdots\ar[r]&H_{\mathcal{B}}^i(\mathcal{X},n)\ar[r]\ar@{=}[d]&H_{\alpha}^i(\mathcal{X},n)\ar[r]\ar[d]&H_{\mathcal{A}}^{i+1}(\mathcal{X},n)\ar[r]^{\Phi}\ar[d]&H_{\mathcal{B}}^{i+1}(\mathcal{X},n)\ar[r]\ar@{=}[d]&\cdots\\
\cdots\ar[r]&H_{\mathcal{B}}^i(\mathcal{X},n)\ar[r]^{\Psi}&H_{\mathcal{C}}^i(\mathcal{X},n)\ar[r]\ar[d]&H_{\beta}^i(\mathcal{X},n)\ar[r]\ar[d]&H_{\mathcal{B}}^{i+1}(\mathcal{X},n)\ar[r]&\cdots\\
&&H_{\syn}^{i+2}(\mathcal{X},n)\ar[d]\ar@{=}[r]&H_{\syn}^{i+2}(\mathcal{X},n)\ar[d]&&\\
\cdots\ar[r]&H_{\mathcal{B}}^{i+1}(\mathcal{X},n)\ar[r]&H_{\alpha}^{i+1}(\mathcal{X},n)\ar[r]\ar[d]&H_{\mathcal{A}}^{i+2}(\mathcal{X},n)\ar[r]\ar[d]&H_{\mathcal{B}}^{i+2}(\mathcal{X},n)\ar[r]&\cdots\\
&&\vdots&\vdots&&
}\]
\end{proposition}

\begin{proposition}[Leray spectral sequence]
If $\mathcal{X}$ satisfies (HK), there exists a spectral sequence
\begin{equation}\label{spectral}
E_2^{i,j}=\Ext^i_{\MF}(K_0,H^j_{\Hdg}(\mathcal{X},n))\Rightarrow H^{i+j}_{\syn}(\mathcal{X},n)
\end{equation}
degenerating at $E_3$.
\end{proposition}

\begin{proof}
With renumbering, the spectral sequence \eqref{spectral} is associated to the canonical filtration of a representing complex of the object in $D^b(\MF)$ corresponding to $\RG_{\Hdg}(\mathcal{X})$ (cf. \cite{Deligne} 1.4.5). Note that $\MF$ does not have any injective objects except for $0$, but $\Ind(\MF)$ has enough injective (\cite{St} Theorem 2.2) and $D^b(\MF)$ is a full subcategory in $\Ind(\MF)$ (\cite{Huber2} Proposition 2.2). Since $E_2^{i,j}=0$ for $i\geq 3$, it degenerates at $E_3$.
\end{proof}

For smooth scheme $\mathcal{X}$ we denote the original rigid syntomic cohomology of Besser by $\widetilde{H}^i_{\syn}(\mathcal{X},n)$.

\begin{proposition}\label{syncomparison}
When $\mathcal{X}$ is smooth over $V$, there exists a direct decomposition
\[H^i_{\syn}(\mathcal{X},n)\cong\widetilde{H}^i_{\syn}(\mathcal{X},n)\oplus H^{i-2}(\Cone(1-\frac{\phi}{p^{n-1}}:\RG_{\rig}(\mathcal{X})\rightarrow\RG_{\rig}(\mathcal{X}))).\]
Here $\RG_{\rig}(\mathcal{X})$ is a (log) rigid complex of $\mathcal{X}$ (for a choice of a family of local embeddings).
In particular, if $\mathcal{X}$ is smooth projective and $2n-i\neq 0,1$, then $H^i_{\syn}(\mathcal{X},n)$ and $\widetilde{H}^i_{\syn}(\mathcal{X},n)$ coinside.
\end{proposition}

\begin{proof}
We write specializations of $\RG_{\Hdg}(\mathcal{X})$ by $M_{\rig}^{\bullet}$, $M_K^{\bullet}$, and $M_{\dR}^{\bullet}$. By Remark \ref{coneremark}, \cite{CCM} Proposition 5.10 and Remark 2.9 (iii), $H^i_{\syn}(\mathcal{X},n)$ and $\widetilde{H}^i_{\syn}(\mathcal{X},n)$ are computed as the cohomology groups of the total complexes of
\begin{eqnarray}\label{cpx1}
M_{\rig}^{\bullet}\oplus F^0M_{\dR}^{\bullet}\rightarrow&M_{\rig}^{\bullet}\oplus M_{\rig}^{\bullet}\oplus M_K^{\bullet}&\rightarrow M_{\rig}^{\bullet}\\
\nonumber(x,y)\mapsto&(0,x-\phi(x),x-y)&\\
\nonumber&(x,y,z)&\mapsto x-p\phi(x)
\end{eqnarray}
and
\begin{eqnarray}\label{cpx2}
M_{\rig}^{\bullet}\oplus F^0M_{\dR}^{\bullet}&\rightarrow&M_{\rig}^{\bullet}\oplus M_K^{\bullet}\\
\nonumber(x,y)&\mapsto&(x-\phi(x),x-y)
\end{eqnarray}
respectively. Since \eqref{cpx1} splits to a direct sum of \eqref{cpx2} and a translation of
\begin{equation}\label{cpx3}
1-p\phi:M_{\rig}^{\bullet}\rightarrow M_{\rig}^{\bullet},
\end{equation}
we have
\begin{eqnarray*}
H^i_{\syn}(\mathcal{X},n)&=&H^i(\Cone\eqref{cpx2}[-1])\oplus H^i(\Cone\eqref{cpx3}[-2])\\
&=&\widetilde{H}^i_{\syn}(\mathcal{X},n)\oplus H^{i-2}(\Cone(1-\frac{\phi}{p^{n-1}}:\RG_{\rig}(\mathcal{X})\rightarrow\RG_{\rig}(\mathcal{X}))).
\end{eqnarray*}
$H^{i-2}(\Cone(1-\frac{\phi}{p^{n-1}}:\RG_{\rig}(\mathcal{X})\rightarrow\RG_{\rig}(\mathcal{X})))$ equals zero if and only if $1-\frac{\phi}{p^{n-1}}$ on $H^{i-2}_{\rig}(\mathcal{X})$ is surjective and on $H^{i-1}_{\rig}(\mathcal{X})$ is injective. So the statement follows from a consequence of the Weil conjecture (cf. \cite{KM}).
\end{proof}

\subsection{Log syntomic cohomology of a simplicial strictly semistable scheme} 
Let $\mathcal{X}_{\bullet}$ be a simplicial strictly semistable scheme over $V$ such that each $\mathcal{X}_i$ satisfies (HK). By the functoriality, we have canonical maps
\[\RG_{\Hdg}(\mathcal{X}_0)\rightarrow\RG_{\Hdg}(\mathcal{X}_1)\rightarrow\RG_{\Hdg}(\mathcal{X}_2)\rightarrow\cdots.\]
in $\pHD$. Let
\begin{equation}\label{seqD}
\RG'_{\Hdg}(\mathcal{X}_0)\rightarrow\RG'_{\Hdg}(\mathcal{X}_1)\rightarrow\RG'_{\Hdg}(\mathcal{X}_2)\rightarrow\cdots
\end{equation}
be the corresponding sequence in $D^b(\MF)$. Fix a representing complex $M_i^{\bullet}$ of $\RG'_{\Hdg}(\mathcal{X}_i)$ for each $i$. For every $i\geq 0$ we can take maps
\[M_i^{\bullet}\rightarrow N_{i+1}^{\bullet}\xleftarrow[]{\cong}M_{i+1}^{\bullet}\]
such that they represent $M^{\bullet}_i\rightarrow M^{\bullet}_{i+1}$, and the second leftwards arrow is quasi-isomorphic.
Put $L_0^{\bullet}=M_0^{\bullet}$ and $L_1^{\bullet}=N_1^{\bullet}$. 
For $i\geq 2$ let $L_i^{\bullet}$ be the quasi-pushout of
\[L_{i-1}^{\bullet}\leftarrow M^{\bullet}_{i-1}\rightarrow N^{\bullet}_i.\]
Then we obtain the sequence
\begin{equation}\label{seqC}
L_0^{\bullet}\rightarrow L^{\bullet}_1\rightarrow L^{\bullet}_2\rightarrow\cdots
\end{equation}
in $C^b(\MF)$ with quasi-isomorphisms $M_i^{\bullet}\rightarrow L_i^{\bullet}$. Note that in $D^b(\MF)$, \eqref{seqC} is isomorphic to \eqref{seqD}.
When we have another sequence
\[L'^{\bullet}_0\rightarrow L'^{\bullet}_1\rightarrow L'^{\bullet}_2\rightarrow\cdots,\]
we have a commutative diagram
\[\xymatrix{
L_0^{\bullet}\ar[r]\ar[d]&L_1^{\bullet}\ar[r]\ar[d]&L_2^{\bullet}\ar[r]\ar[d]&\cdots\\
L''^{\bullet}_0\ar[r]&L''^{\bullet}_1\ar[r]&L''^{\bullet}_2\ar[r]&\cdots\\
L'^{\bullet}_0\ar[r]\ar[u]&L'^{\bullet}_1\ar[r]\ar[u]&L'^{\bullet}_2\ar[r]\ar[u]&\cdots,
}\]
here $L''^{\bullet}_i$ is the quasi-push out of the diagram $L^{\bullet}_i\leftarrow M^{\bullet}_i\rightarrow L'^{\bullet}_i$, and the vertical arrows are quasi-isomorphic. Therefore the image of the total complex of \eqref{seqC} in $D^b(\MF)$ is independent of all choices, we denote that by $\RG'_{\Hdg}(\mathcal{X}_{\bullet})$.
We define the log rigid syntomic cohomology of $\mathcal{X}_{\bullet}$ by
\[H^i_{\syn}(\mathcal{X}_{\bullet},n)=\Hom_{D^b(\MF)}(K_0,\RG'_{\Hdg}(\mathcal{X}_{\bullet})(n)[i]).\]

\subsection{Chern class map}
Let $\mathcal{X}$ be a strictly semistable scheme satisfying (HK). We construct the functorial Chern class maps
\[c_{\syn}:K_i(\mathcal{X})\rightarrow H^{2j-i}_{\syn}(\mathcal{X},j)\]
in the similar way to \cite{Besser} Theorem 7.5 based on Huber's method.

Let
\[c^{\mathrm{univ}}_j\in\widetilde{H}^{2j}_{\syn}(\BGL,j)=\varinjlim_n\widetilde{H}^{2j}_{\syn}(\BGL_n,j)\]
be the universal Chern class constructed in the section 7 of \cite{Besser}. We also write $c^{\mathrm{univ}}_j$ its image in $H^{2j}_{\syn}(\BGL,j)$ through the canonical inclusion given by Proposition \ref{syncomparison}.

For $i\geq 0$ we have maps
\begin{equation}\label{map1}
K_i(\mathcal{X})\xleftarrow[]{\cong}\varinjlim_{\mathcal{U}_{\bullet}}\pi_i\Tot(\cZ\times\cZ_{\infty}\BGL(\mathcal{U}_{\bullet}))\rightarrow\varinjlim_{\mathcal{U}_{\bullet}}\pi_i\Tot\cZ_{\infty}\BGL(\mathcal{U}_{\bullet})
\end{equation}
where $\mathcal{U}_{\bullet}$ runs through all finite open covering of $\mathcal{X}$ considered as affine simplicial schemes, and the first leftwards arrow is isomorphic.

For a simplicial cosimplicial abelian group $\{A_{m,n}\}_{m,n\in\cN}$, write the associated complex by $\Ass(A_{\bullet,\bullet})$. Let
\[\Hom(\mathcal{U}_m,\mathrm{B}_n\GL)=\varinjlim_{r}\Hom(\mathcal{U}_m,\mathrm{B}_n\GL_r).\]
Then there are natural maps of simplicial sets
\[\cZ_{\infty}[\BGL(\mathcal{U}_n)]\rightarrow\cZ\cZ_{\infty}[\BGL(\mathcal{U}_n)]\xleftarrow[]{\cong}\cZ[\BGL(\mathcal{U}_n)]\rightarrow \cQ[\Hom(\mathcal{U}_n,\mathrm{B}_{\bullet}\GL)]\]
which induce a map
\begin{equation}\label{map2}
\pi_i\Tot\cZ_{\infty}\BGL(\mathcal{U}_{\bullet})\rightarrow H^{-i}(\Ass(\cQ[\Hom(\mathcal{U}_{\bullet},\mathrm{B}_{\bullet}\GL)])).
\end{equation}

For objects $L^{\bullet}$ and $M^{\bullet}$ in $C^b(\MF)$ we have a map of complexes 
\[\cQ[\Hom_{C^b(\MF)}(L^{\bullet},M^{\bullet})]\rightarrow\Gamma^{\bullet}(L^{\bullet},M^{\bullet})[-2]\]
sending $f\in\Hom_{C^b(\MF)}(L^{\bullet},M^{\bullet})$ to
\begin{eqnarray*}(f,f,0,0,0,0)&\in&\iHom^0(L^{\bullet},M^{\bullet})\oplus F^0\iHom^0(L_K^{\bullet},M_K^{\bullet})\oplus\iHom^{-1}(L^{\bullet},M^{\bullet})\\
&&\oplus\iHom^{-1}(L^{\bullet},M^{\bullet})\oplus\iHom^{-1}(L_K^{\bullet},M_K^{\bullet})\oplus\iHom^{-2}(L^{\bullet},M^{\bullet})\\
&=&\Gamma^{-2}(L^{\bullet},M^{\bullet}).
\end{eqnarray*}
By taking direct limit for quasi-isomorphisms $M^{\bullet}\rightarrow M'^{\bullet}$, we obtain a map
\[\cQ[\Hom_{D^b(\MF)}(L^{\bullet},M^{\bullet})]\rightarrow\varinjlim_{M^{\bullet}\rightarrow M'^{\bullet}:\text{qis}}\Gamma^{\bullet}(L^{\bullet},M'^{\bullet})[-2]\cong\Gamma^{\bullet}(L^{\bullet},M^{\bullet})[-2]\]
in $D^b(\Ab)$.
Since the total complex of the double complex associated to
\[\{\Gamma^{\bullet}(\RG'_{\Hdg}(\mathrm{B}_n\GL_r),\RG'_{\Hdg}(\mathcal{U}_m))\}_{m,n}\]
is equal to $\Gamma^{\bullet}(\RG'_{\Hdg}(\BGL_r),\RG'_{\Hdg}(\mathcal{U}_{\bullet}))$, the maps of complexes
\begin{eqnarray*}
\cQ[\Hom(\mathcal{U}_m,\mathrm{B}_n\GL_r)]&\rightarrow&\cQ[\Hom_{D^b(\MF)}(\RG'_{\Hdg}(\mathrm{B}_n\GL_r),\RG'_{\Hdg}(\mathcal{U}_m))]\\
&\rightarrow&\Gamma^{\bullet}(\RG'_{\Hdg}(\mathrm{B}_n\GL_r),\RG'_{\Hdg}(\mathcal{U}_m))\end{eqnarray*}
induce
\[\Ass(\cQ[\Hom(\mathcal{U}_{\bullet},\BGL_r)])\rightarrow\Gamma^{\bullet}(\RG'_{\Hdg}(\BGL_r),\RG'_{\Hdg}(\mathcal{U}_{\bullet}))\]
in $D^b(\Ab)$. Then we get a map
\begin{eqnarray}\label{map3}
H^{-i}(\Ass(\cQ[\Hom(\mathcal{U}_{\bullet},\BGL)]))&\rightarrow&\varinjlim_r\Ext^{-i}(\RG'_{\Hdg}(\BGL_r),\RG'_{\Hdg}(\mathcal{U}_{\bullet}))\\
\nonumber&=&\varinjlim_r\Hom_{D^b(\MF)}(\RG'_{\Hdg}(\BGL_r),\RG'_{\Hdg}(\mathcal{U}_{\bullet})[-i])
\end{eqnarray}

Composing \eqref{map1}, \eqref{map2} and \eqref{map3}, we obtain
\[\alpha:K_i(\mathcal{X})\rightarrow\varinjlim_r\Hom_{D^b(\MF)}(\RG'_{\Hdg}(\BGL_r),\RG'_{\Hdg}(\mathcal{X})[-i])\]
since we have a natural isomorphism $\RG_{\Hdg}'(\mathcal{X})\rightarrow\RG_{\Hdg}'(\mathcal{U}_{\bullet})$.
We define the syntomic Chern class map
\[c_{\syn}:K_i(\mathcal{X})\rightarrow H^{2j-i}_{\syn}(\mathcal{X},j)\]
by assigning $\alpha(x)(c_j^{\mathrm{univ}})$ to $x$.

\end{document}